\documentclass[11pt]{amsart}
\usepackage{epigraph}
\usepackage{amsthm}
\usepackage{bm}
\usepackage{tikz-cd}
\usepackage{amsmath}
\usepackage{amssymb}
\usepackage{relsize}
\usepackage{amsfonts}
\usepackage{amsxtra}
\usepackage{braket}
\usepackage{lmodern}
\usepackage{mathrsfs}
\usepackage{array}
\usepackage{xcolor}
\usepackage{mathtools}
\definecolor{citecol}{rgb}{0.07,0.07,0.05}
\definecolor{urlcol}{rgb}{0.06,0.04,0.09}
\definecolor{linkcol}{rgb}{0.01,0.03,0.08}
\usepackage[colorlinks,linkcolor=linkcol,urlcolor=urlcol,citecolor=citecol,pagebackref,breaklinks]{hyperref}
\usepackage{appendix} 
\usepackage[lite,abbrev,msc-links,alphabetic]{amsrefs}
\usepackage{indentfirst}
\usepackage{mathtools}
\usepackage[all]{xy}
\allowdisplaybreaks
 \numberwithin{equation}{subsection}
\theoremstyle{plain}
\newtheorem{theorem}{Theorem}[section]
\newtheorem{lemma}[theorem]{Lemma}
\newtheorem{proposition}[theorem]{Proposition}
\newtheorem{claim}[theorem]{Claim}
\newtheorem{corollary}[theorem]{Corollary}
\newtheorem{conjecture}[theorem]{Conjecture}

\theoremstyle{definition}
\newtheorem{definition}[theorem]{Definition}
\newtheorem{remark}[theorem]{Remark}
\newtheorem{example}[theorem]{Example}

\theoremstyle{remark}

\newcommand{\BF}{{\mathbb F}}

\newcommand{\BL}{{\mathbb L}}

\newcommand{\BN}{{\mathbb N}}

\newcommand{\BP}{{\mathbb P}}
\newcommand{\BQ}{{\mathbb Q}}

\newcommand{\BV}{{\mathbb V}}

\newcommand{\BX}{{\mathbb X}}
\newcommand{\BY}{{\mathbb Y}}
\newcommand{\BZ}{{\mathbb Z}}

\newcommand{\CN}{{\mathcal N}}

\newcommand{\CY}{{\mathcal Y}}
\newcommand{\CZ}{{\mathcal Z}}

\newcommand{\Fb}{{\mathfrak b}}

\DeclareMathOperator{\Lie}{Lie}
\DeclareMathOperator{\Charpol}{Charpol}
\DeclareMathOperator{\End}{End}

\DeclareMathOperator{\Spf}{Spf}

\DeclareMathOperator{\Hom}{Hom}

\DeclareMathOperator{\val}{val}

\DeclareMathOperator{\Mod}{mod}
\DeclareMathOperator{\diag}{diag}

\author{Sungyoon Cho}
\address[Sungyoon Cho]{Department of Mathematics, University of Arizona}
\email{sungyooncho@math.arizona.edu}

\title[Representation densities II]{On local representation densities of hermitian forms and special cycles II}
\date{\today}

\begin{document}

\begin{abstract}
In this paper, we prove that there are certain relations among representation densities and provide an efficient way to compute representation densities by using these relations. As an application, we compute some arithmetic intersection numbers of special cycles on unitary Shimura varieties and propose a conjecture on these.
\end{abstract}

\maketitle
\tableofcontents{}

\section{Introduction}
In \cite{KR2} and \cite{KR3}, Kudla and Rapoport made a precise conjectural formula for the arithmetic intersection numbers of special cycles on unitary Shimura varieties with hyperspecial level structure. This is called the Kudla-Rapoport conjecture. This is proved by Li and Zhang in 2019 \cite{LZ}. In \cite{Cho2}, we made a variant of the Kudla-Rapoport conjecture in the case that the unitay Shimura variety has minuscule parahoric level structures. This conjectural formula predicts a relation between the arithmetic intersection numbers of special cycles on unitay Rapoport-Zink spaces and the derivative of some representation densities of hermitian forms. To prove this conjecture, one may hope to compute the arithmetic intersection numbers of special cycles directly and compare these with the analytic side. However, there is currently no way to compute the geometric side directly except some low dimensional cases. For this reason, we hope to get an idea from the analytic side of the conjecture since it is at least possible to compute.

The analytic side of the Kudla-Rapoport conjecture is related to representation densities of hermitian forms. Therefore, we would like to test our geometric guesses by computing some representation densities. Let $p$ be an odd prime and let $F$ be an unramified extension of $\BQ_p$ with ring of integers $O_F$ and residue field $\BF_q$. Let $E$ be a quadratic unramified extension of $F$ with ring of integers $O_E$. For $n \times n$ hermitian matrices $A$ and $B$, we define the representation density $\alpha(A,B)$ by
	\begin{equation*}
	\alpha(A,B)=\lim_{d \rightarrow \infty} (q^{-d})^{n(2m-n)}\vert \lbrace X \in M_{m,n}(O_E/\pi^dO_E)\vert A[X]\equiv B (\Mod \pi^d)\rbrace \vert.
\end{equation*}

 We have an explicit formula for these representation densities in \cite{Hir2}, however, the actual computation is very complicated.
For example, for $A=A_{40}:=\left(\begin{array}{cc}p^4 & 0 \\ 0 & 1 \end{array}\right)$ and $B=A_{42}:=\left(\begin{array}{cc}p^4 & 0 \\ 0 & p^2 \end{array}\right)$, we used several pages to compute that $\alpha(A_{40},A_{42})=q^6(1+\dfrac{1}{q})^2$ with the above formula. Therefore, we need a more efficient way to compute representation densities.

In the present paper, we prove that there are certain relations among representation densities and by using this, we give an efficient way to compute representation densities. We also apply this method to compute some arithmetic intersection numbers of special cycles and then give a conjecture on these.

We now describe our results in more detail. We define the set of hermitian matrices $X_n(E)=\lbrace X \in GL_{n}(E) \vert ^tX^*=X \rbrace$, and $X_n(O_E)=X_n(E) \cap M_{n,n}(O_E)$. Let $\Lambda_{n}^+$ be the set of $\lambda=(\lambda_1,\lambda_2,\dots,\lambda_n)$ such that $\lambda_1\geq \dots \geq \lambda_n \geq 0$. For $\lambda \in \Lambda_n^+$, we write $A_{\lambda}$ for $\diag(p^{\lambda_1},\dots,p^{\lambda_n})$. Then, the complete set of representatives of $GL_{n}(O_E)$-equivalence classes in $X_{n}(O_E)$ is given by $\lbrace \lambda \in \Lambda_n^+ \vert A_{\lambda} \rbrace$.

Let $\Lambda_{n,s}^+$ be the subset of $\Lambda_n^+$ such that
	\begin{equation*}
	\Lambda_{n,s}^+=\lbrace \xi=(\xi_1,\dots,\xi_n) \in \Lambda_n^+ \text{  }\vert \text{  }\xi_{n-s} \geq 1, \xi_{n-s+1}=\dots=\xi_{n}=0 \rbrace.
\end{equation*}
For $\xi \in \Lambda_{n,s}^+$, we define $\xi_{a,b}^+$ as
\begin{equation*}
	\xi_{a,b}^+:=(\xi_1,\dots,\xi_{n-s},\overset{a}{\overbrace{2,\dots,2}},\overset{b}{\overbrace{1,\dots,1}},\overset{s-a-b}{\overbrace{0,\dots,0}})
\end{equation*}
For integers $1 \leq s \leq n$, $0 \leq i \leq s$, we define $d_{n,s,i}$ by the coefficients of the polynomial
\begin{equation*}\begin{array}{ll}
		\sum_{i=0}^s d_{n,s,i}X^i\\\\
		=(1-(-q)^{-n}X)(1-(-q)^{-n+1}X)\dots(1-(-q)^{-n+s-1}X)
	\end{array}
\end{equation*}

Then, our main theorem is as follows.
\begin{theorem}\label{theorem1.1}(Theorem \ref{theorem2.5})
	Let $s$ and $n$ be integers such that $1 \leq s \leq n$, and let $\xi \in \Lambda_{n,s}^+$. Then, for any $B \in X_n(O_E)$ such that $\pi^{-1}B$ is also integral (i.e., the representative of $B$ is $A_{\lambda}$ for some $\lambda \in \Lambda_n^{+}$ such that $\lambda \geq (1,\dots,1)$), we have the following equality.
	\begin{equation*}
		\begin{array}{l}
			\mathlarger{\mathlarger{\sum_{i=0}^{s} d_{n,s,i}\alpha(A_{\xi_{0,i}^+},B)-\dfrac{(-1)^s}{(-q)^{ns}}\sum_{i=0}^{s}d_{n,s,i}\alpha(A_{\xi_{i,s-i}^+},B)=0}}.
		\end{array}
	\end{equation*}
\end{theorem}

For example, when $n=2$, the theorem implies that for any $B$ such that $\pi^{-1}B$ is integral and $\lambda \geq 1$, we have\begin{equation*}\begin{array}{l}
\alpha(A_{\lambda0},B)=\dfrac{1}{q^4}\alpha(A_{\lambda2},B)
	\end{array}
\end{equation*}

Therefore, we can simply compute that
\begin{equation*}
	\alpha(A_{40},A_{42})=\dfrac{1}{q^4}\alpha(A_{42},A_{42})=q^4 \alpha(A_{20},A_{20})=q^4(1+\dfrac{1}{q})\alpha(A_2,A_2)=q^6(1+\dfrac{1}{q})^2.
\end{equation*}

Our theorem is true for arbitrary $n$. Therefore, this can be applied to compute any representation densities.

When we apply the theorem to compute representation densities, it is better to use normalized representation densities and a normalized version of the theorem. Also, the Cho-Yamauchi's formula for the analytic part of the Kudla-Rapoport conjecture (\cite[Theorem 3.5.1]{LZ}, \cite[Theorem 4.7]{Cho3}) is formulated in terms of normalized representation densities.

For $A, B \in X_n(O_E)$, we define the normalized representation density $A(B)$ as
\begin{equation*}
	A(B):=\dfrac{\alpha(A,B)}{\alpha(A,A)}.
\end{equation*}
Then this has the following reductions.
\begin{proposition}\label{proposition1.2}(Proposition \ref{proposition2.14})
	For $\xi, \lambda \in \Lambda_{n-1}^+$, we consider $(\xi,0), (\lambda,0)$ in $\Lambda_n^+$. Then, we have
	\begin{equation*}
		A_{\xi,0}(A_{\lambda,0})=A_{\xi}(A_{\lambda}).
	\end{equation*}
\end{proposition}

\begin{proposition}\label{proposition1.3}(Proposition \ref{proposition2.15})
	For $\xi, \lambda$ in $\Lambda_n^+$, such that $\xi, \lambda \geq (1,1\dots,1)$, we consider
	\begin{equation*}\begin{array}{l}
			\xi-1:=(\xi_1-1,\dots,\xi_{n}-1) \in \Lambda_n^+\\
			\lambda-1:=(\lambda_1-1,\dots,\lambda_n-1) \in \Lambda_n^+.
		\end{array}
	\end{equation*}
	Then, we have \begin{equation*}
		A_{\xi}(A_{\lambda})=A_{\xi-1}(A_{\lambda-1}).
	\end{equation*}
\end{proposition}
Also, Theorem \ref{theorem1.1} gives the following corollaries when $n=2$ and $n=3$.
\begin{corollary}[$n=2$]\label{corollary1.3}
	Assume that $\lambda \geq 3$. For any $B \in X_2(O_E)$ such that $\pi^{-1}B$ is integral, the following equalities hold.
	\begin{enumerate}
		\item $A_{00}(B)=(q+1)A_{11}(B)-qA_{22}(B)$.
		\item $A_{10}(B)=A_{21}(B)$.
		\item $A_{20}(B)=q(q-1)A_{22}(B)$.
		\item $A_{\lambda0}(B)=q^2A_{\lambda2}(B)$.
	\end{enumerate}
\end{corollary}
\begin{corollary}[$n=3$]\label{corollary1.4}(Corollary \ref{corollary2.19})
	Assume that $\lambda, \kappa \geq 3$. For any $B \in X_3(O_E)$ such that $\pi^{-1}B$ is integral, the following equalities hold.
	\begin{enumerate}
		\item $A_{000}(B)=(q+1)A_{211}(B)-q^3A_{222}(B)$.
		
		\item $A_{100}(B)=(q^3+1)A_{111}(B)-qA_{221}(B)$.
		\item $A_{110}(B)=A_{211}(B)$.
		\item $A_{\lambda\kappa0}(B)=q^4A_{\lambda\kappa2}(B)$.
		\item $A_{\lambda20}(B)=q^3(q-1)A_{\lambda22}(B)$.
		\item $A_{220}(B)=q^2(q^2-q+1)A_{222}(B)$.
		\item $A_{\lambda00}(B)=q^2(q+1)A_{\lambda11}(B)-q^5A_{\lambda22}(B)$.
		\item $A_{200}(B)=q^2(q+1)A_{211}(B)-q^3(q^2-q+1)A_{222}(B)$.
		\item $A_{210}(B)=q(q-1)A_{221}(B)$.
		\item $A_{\lambda10}(B)=q^2A_{\lambda21}(B)$.
	\end{enumerate}
\end{corollary}
 In Corollary \ref{corollary2.20}, we give a similar corollary of Theorem \ref{theorem1.1} in case where $n=4$.

These corollaries can be used to compute normalized representation densities. The most important advantage of this method is that we can reduce the computations of representation densities for $B$ to the ones for $\pi^{-1}B$. For example, consider $A_{300}(A_{832})$. By Corollary \ref{corollary1.3}, Corollary \ref{corollary1.4}, Proposition \ref{proposition1.2}, and Proposition \ref{proposition1.3}, we can compute that
\begin{equation*}\begin{array}{ll}
	A_{300}(A_{832})&=q^2(q+1)A_{311}(A_{832})-q^5A_{322}(A_{832})\\
	&=q^2(q+1)A_{200}(A_{721})-q^5A_{100}(A_{610})\\
	&=q^4(q+1)^2A_{211}(A_{721})-q^5A_{10}(A_{61})\\
	&=q^4(q+1)^2A_{100}(A_{610})-q^5A_{21}(A_{61})\\
	&=q^4(q+1)^2A_{10}(A_{61})-q^5A_{10}(A_{50})\\
	&=\lbrace q^4(q+1)^2-q^5 \rbrace A_{1}(A_5)\\
	&=q^6+q^5+q^4.
	\end{array}
\end{equation*}

Another important advantage of this method is that we do not need to compute all representation densities separately: when we compute a representation density, we can use other representation densities. We refer to Example \ref{example3.4} for this.

In Section \ref{section3}, we apply Theorem \ref{theorem1.1} to compute some arithmetic intersection numbers of special cycles on unitary Rapoport-Zink spaces. In fact, the arithmetic intersection numbers of special cycles on the Drinfelf upper half plane are the main motivation of our Theorem \ref{theorem1.1}. We refer the reader to Example \ref{example3.3} for this. 

Our computational method allows us to test some geometric guesses on the arithmetic intersection numbers of special cycles. In Example \ref{example3.4}, we test our geometric guess by using some computations of representation densities. Using this example, we make Conjecture \ref{conjecture}. This is modelled on the analytic part of \cite{LZ} (in particular \cite[Theorem 4.2.1, Theorem 5.4.1]{LZ}). We believe that we can use our computations to make more precise conjectures on the arithmetic intersection numbers of special cycles. We refer the reader to Remark \ref{remark} for this.

\section{Relations among representation densities}
In this section, we prove some relations among representation densities.

\subsection{Representation densities}
In this subsection, we will recall the definition of representation densities and their formulas. 

We fix a prime $p >2$. Let $F$ be a finite extension of $\BQ_p$ with ring of integers $O_F$ and residue field $\BF_q$. We fix a uniformizer $\pi$ of $O_F$. Let $E$ be a quadratic unramified extension of $F$ with ring of integers $O_E$. Denote by $^*$ the nontrivial Galois automorphism of $E$ over $F$.

We define the set of hermitian matrices
\begin{equation*}
	\begin{array}{l}
	X_n(E)=\lbrace X \in GL_{n}(E) \vert ^tX^*=X \rbrace.
	\end{array}
\end{equation*}
We define $X_n(O_E)=X_n(E) \cap M_{n,n}(O_E)$ and $K_n=GL_n(O_E)$.

For $g \in GL_n(E)$, and $X \in X_n(E)$, we define the group action of $GL_n(E)$ on $X_n(E)$ by $g \cdot X=gX^tg^*$.

For $X \in M_{m,n}(E)$ and $A \in X_m(E)$, we denote by $A[X]=^tX^*AX$.

Let $\Lambda_n^+:=\lbrace \lambda=(\lambda_1, \dots, \lambda_n) \in \BZ^n \vert \lambda_1 \geq \lambda_2 \geq \dots \geq \lambda_n \geq 0 \rbrace$. For $\lambda \in \Lambda_n^+$, we define $A_{\lambda}:=\diag(\pi^{\lambda_1},\pi^{\lambda_2},\dots,\pi^{\lambda_n})$.

Note that the complete set of representatives of $K_n \backslash X_n(O_E)$ is given by the set of diagonal matrices $A_{\lambda}$ for $\lambda=(\lambda_1,\lambda_2,\dots,\lambda_n) \in \Lambda_n^+$.

\begin{definition}
	For $A \in X_m(O_E)$ and $B \in X_n(O_E)$, we define $\alpha(A,B)$ by
	\begin{equation*}
		\alpha(A,B)=\lim_{d \rightarrow \infty} (q^{-d})^{n(2m-n)}\vert \lbrace X \in M_{m,n}(O_E/\pi^dO_E)\vert A[X]\equiv B (\Mod \pi^d)\rbrace \vert.
	\end{equation*}
\end{definition}

Now, we recall an explicit formula for the representation densities from \cite{Hir2}.

For each $\mu \in (\mu_1, \dots, \mu_n) \in \Lambda_n^+$, we denote by
\begin{equation*}
	\begin{array}{l}
		\tilde{\mu}:=(\mu_{1}+1, \dots, \mu_n+1); \\\\
		\vert \mu \vert:=\sum_{i=1}^{n}  \mu_n;\\\\
		n(\mu):=\sum_{i=1}^{n}(i-1)\mu_i; \\\\
		\mu_i':= \vert \lbrace j \vert \mu_j \geq i \rbrace \vert, \text{for }i \geq 1.\\
	\end{array}
\end{equation*}

Also, for $\mu \in \Lambda_n^+$, $\xi \in \Lambda_m^+$, we write
\begin{equation*}
	\langle \xi', \mu' \rangle:=\sum_{i \geq 1} \xi_i'\mu_i'.
\end{equation*}

For $\mu$ and  $\lambda$ in $\Lambda_n^+$, we denote by $\mu \leq \lambda$ if $\mu_i \leq \lambda_i$ for all $1 \leq i \leq n$.

For nonnegative integers $u \geq v \geq 0$, we define
\begin{equation*}
	\left[ \begin{array}{l}
		u\\
		v
	\end{array}\right]:=\dfrac{\prod_{i=1}^u (1-(-q)^{-i})}{\prod_{i=1}^v (1-(-q)^{-i}) \prod_{i=1}^{u-v} (1-(-q)^{-i})}
\end{equation*}

Finally, for $\mu, \lambda \in \Lambda_n^+$, we define
\begin{equation*}
	I_j(\mu, \lambda):=\sum_{i=\mu_{j+1}'}^{\min((\tilde{\lambda})'_{j+1}, \mu_j')}(-q)^{i(2(\tilde{\lambda})'_{j+1}+1-i)/2}\left[ \begin{array}{l}  
		(\tilde{\lambda})_{j+1}'-\mu_{j+1}'\\
		  (\tilde{\lambda})_{j+1}'-i
	\end{array}\middle] \middle[ \begin{array}{l} (\tilde{\lambda})_{j}'-i\\ (\tilde{\lambda})_{j}'-\mu_{j}' \end{array}\right].
\end{equation*}

Now, we can state the following formula for representation densities.

\begin{proposition}\label{proposition2.2}
	(\cite[Theorem II]{Hir2}) For $\lambda \in \Lambda_n^+$ and $\xi \in \Lambda_m^+$ with $m \geq n$, we have
	\begin{equation*}
		\alpha(A_{\xi},A_{\lambda})=\sum_{\substack{\mu \in \Lambda_n^+\\ \mu \leq \tilde{\lambda}}}(-1)^{\vert \mu \vert}(-q)^{-n(\mu)+(n-m-1)\vert \mu \vert+\langle \xi',\mu' \rangle} \prod_{j \geq 1} I_j(\mu,\lambda).
	\end{equation*}
\end{proposition}

\subsection{Relations among representation densities} 

In this subsection, we will prove some relations among representation densities.

\begin{definition} We fix an integer $n \geq 1$.
For integers $1 \leq s \leq n$, $0 \leq i \leq s$ we define $d_{n,s,i}$ by the coefficients of the polynomial
\begin{equation*}\begin{array}{ll}
	\sum_{i=0}^s d_{n,s,i}X^i\\\\
	=(1-(-q)^{-n}X)(1-(-q)^{-n+1}X)\dots(1-(-q)^{-n+s-1}X)
	\end{array}
\end{equation*}

For example, $d_{n,k,0}=1$ for any $k$, $d_{n,1,1}=-\dfrac{1}{(-q)^{n}}$, $d_{n,2,1}=-\dfrac{(-q)+1)}{(-q)^n}$, $d_{n,2,2}=\dfrac{1}{(-q)^{2n-1}}$.
\end{definition}

\begin{definition} Let $s$ and $n$ be integers such that $1 \leq s \leq n$.
	\begin{enumerate}
	\item We define the subset $\Lambda_{n,s}^+$ of $\Lambda_n^+$ as
	\begin{equation*}
		\Lambda_{n,s}^+=\lbrace \xi=(\xi_1,\dots,\xi_n) \in \Lambda_n^+ \text{  }\vert \text{  }\xi_{n-s} \geq 1, \xi_{n-s+1}=\dots=\xi_{n}=0 \rbrace.
	\end{equation*}

\item 	For $\xi \in \Lambda_{n,s}^+$ and $a,b \geq 0$ such that $a+b \leq s$, we write $\xi_{a,b}^+$ for the element in $\Lambda_n^+$ representing
\begin{equation*}
	\xi_{a,b}^+:=(\xi_1,\dots,\xi_{n-s},\overset{a}{\overbrace{2,\dots,2}},\overset{b}{\overbrace{1,\dots,1}},\overset{s-a-b}{\overbrace{0,\dots,0}})
\end{equation*}
In other words, we replace $a$ $0$s by $a$ $2$s, $b$ $0$s by $b$ $1$s.

For example, for $\xi=(4,3,0,0) \in \Lambda_{4,2}^+$, we have that
\begin{equation*}
	\begin{array}{l}
		\xi_{0,1}^+=(4,3,1,0)\\
		\xi_{1,1}^+=(4,3,2,1)\\
		\xi_{2,0}^+=(4,3,2,2),
	\end{array}
\end{equation*}
and for $\xi=(4,1,0,0)\in \Lambda_{4,2}^+$, we have,
\begin{equation*}
	\begin{array}{l}
		\xi_{0,1}^+=(4,1,1,0)\\
		\xi_{1,1}^+=(4,2,1,1)\\
		\xi_{2,0}^+=(4,2,2,1).
	\end{array}
\end{equation*}
\end{enumerate}
\end{definition}

\quad\\

Now, we can state our main theorem as follows.

\begin{theorem}\label{theorem2.5}
	Let $s$ and $n$ be integers such that $1 \leq s \leq n$, and let $\xi \in \Lambda_{n,s}^+$. Then, for any $B \in X_n(O_E)$ such that $\pi^{-1}B$ is also integral (i.e., the representative of $B$ in $K_n\backslash X_n(O_E)$ is $A_{\lambda}$ for some $\lambda \in \Lambda_n^{+}$ such that $\lambda \geq (1,\dots,1)$), we have the following equality.
	\begin{equation*}
		\begin{array}{l}
			\mathlarger{\mathlarger{\sum_{i=0}^{s} d_{n,s,i}\alpha(A_{\xi_{0,i}^+},B)-\dfrac{(-1)^s}{(-q)^{ns}}\sum_{i=0}^{s}d_{n,s,i}\alpha(A_{\xi_{i,s-i}^+},B)=0}}.
		\end{array}
	\end{equation*}
\end{theorem}		\quad\\

	\begin{example}
		\begin{enumerate}
		
	\item For any $B$ in $X_2(O_E)$ such that $\pi^{-1}B$ is integral, we have that
	\begin{equation*}\begin{array}{l}
			\lbrace \alpha(A_{10},B)-\dfrac{1}{(-q)^2}\alpha(A_{11},B)\rbrace\\
			 +\dfrac{1}{(-q)^2}\lbrace \alpha(A_{11},B)-\dfrac{1}{(-q)^2}\alpha(A_{21},B)\rbrace=0.
		
		\end{array}
	\end{equation*}
Therefore, $\alpha(A_{10},B)=\dfrac{1}{(-q)^{4}}\alpha(A_{21},B)$.
\quad\\
	
	\item For any $B$ in $X_2(O_E)$ such that $\pi^{-1}B$ is integral,
	\begin{equation*}\begin{array}{l}
		\alpha(A_{00},B)-\dfrac{(-q)+1}{(-q)^2}\alpha(A_{10},B)+\dfrac{1}{(-q)^3}\alpha(A_{11},B)\\\\
		=\dfrac{1}{(-q)^4}\lbrace 	\alpha(A_{11},B)-\dfrac{(-q)+1}{(-q)^2}\alpha(A_{21},B)+\dfrac{1}{(-q)^3}\alpha(A_{22},B) \rbrace
		\end{array}
	\end{equation*}
\quad\\
	
	\item For $s=1$ and $\xi=(\overline{\xi},0) \in \Lambda_{n,1}^+$ and for any $B$ such that $\pi^{-1}B$ is integral, we have that
\begin{equation*}\begin{array}{l}
		\lbrace \alpha(A_{\overline{\xi},0},B)-\dfrac{1}{(-q)^n}\alpha(A_{\overline{\xi},1},B)\rbrace\\
		+\dfrac{1}{(-q)^n}\lbrace \alpha(A_{\overline{\xi},1},B)-\dfrac{1}{(-q)^n}\alpha(A_{\overline{\xi},2},B)\rbrace=0\\
		
	\end{array}
\end{equation*}
Therefore, $\alpha(A_{\overline{\xi},0},B)=\dfrac{1}{(-q)^{2n}}\alpha(A_{\overline{\xi},2},B)$
	\end{enumerate}
	\end{example}

To prove Theorem \ref{theorem2.5}, we need to use Proposition \ref{proposition2.2}. Let us introduce the following definition.

\begin{definition} Let $l$ and $n$ be integers such that $0 \leq l \leq n$.
	\begin{enumerate}
		\item We define the subset $\Lambda_{n-l}^{2+}$ of $\Lambda_{n-l}^+$ as
		\begin{equation*}
			\Lambda_{n-l}^{2+}=\lbrace \overline{\mu}=(\overline{\mu}_1,\dots,\overline{\mu}_{n-l}) \in \Lambda_{n-l}^+ \text{ }\vert \text{ } \overline{\mu}_{i} \geq 2 \rbrace.
		\end{equation*}
		
		\item For integers $0 \leq k \leq l$ and an element $\overline{\mu_l}=(\overline{\mu}_1,\dots,\overline{\mu}_{n-l}) \in \Lambda_{n-l}^{2+}$ (i.e. $\mu_{l,i} \geq 2$), we define
		\begin{equation*}
			\mu_{l,k}:=(\overline{\mu_l},\overset{k}{\overbrace{1,\dots,1}},\overset{l-k}{\overbrace{0,\dots,0}}) \in \Lambda_n^+.
		\end{equation*}
	\end{enumerate}
\end{definition}
\quad\\

By Proposition \ref{proposition2.2}, we have that
	\begin{equation*}
		\begin{array}{l}
		\alpha(A_{\xi_{a,b}^+},B)=\alpha(A_{\xi_{a,b}^+},A_{\lambda})\\\\
		=\mathlarger{\sum}_{\substack{\mu \in \Lambda_n^+\\ \mu \leq \tilde{\lambda}}}(-1)^{\vert \mu \vert}(-q)^{-n(\mu)-\vert \mu \vert+\langle (\xi_{a,b}^+)',\mu' \rangle} \prod_{j \geq 1} I_j(\mu,\lambda).
		\end{array}
	\end{equation*}

Let us further decompose this as follows.

	\begin{equation}\label{eq2.2.1}\begin{array}{l}
\alpha(A_{\xi_{a,b}^+},A_{\lambda})=\mathlarger{\sum}_{\substack{\mu \in \Lambda_n^+\\ \mu \leq \tilde{\lambda}}}(-1)^{\vert \mu \vert}(-q)^{-n(\mu)-\vert \mu \vert+\langle (\xi_{a,b}^+)',\mu' \rangle} \prod_{j \geq 1} I_j(\mu,\lambda)\\\\

=\mathlarger{\sum}_{l=0}^{n}\mathlarger{\sum}_{\substack{\overline{\mu}_l \in \Lambda_{n-l}^{2+}\\ \overline{\mu}_l \leq \tilde{\lambda}}}\mathlarger{\sum}_{k=0}^{l}

(-1)^{\vert \mu_{k,l} \vert}(-q)^{-n(\mu_{k,l})-\vert \mu_{k,l} \vert+\langle (\xi_{a,b}^+)',\mu_{k,l}' \rangle} \prod_{j \geq 1} I_j(\mu_{k,l},\lambda).\\
\end{array}
\end{equation}

Here, $\overline{\mu}_l \leq \tilde{\lambda}$ means that $\overline{\mu}_1 \leq \tilde{\lambda_1}=\lambda_1+1$, $\dots$, $\overline{\mu}_{n-l} \leq \tilde{\lambda}_{n-l}=\lambda_{n-l}+1$. The condition that $\mu_{k,l} \leq \tilde{\lambda}$ automatically follows from this since $\lambda_i+1$ is greater than $1$ for all $i$.

Now, let us introduce the following definition.
\begin{definition}
	For $\overline{\mu_l}=(\overline{\mu}_1,\dots,\overline{\mu}_{n-l}) \in \Lambda_{n-l}^{2+}$, we define
	\begin{equation*}\begin{array}{l}
		\alpha(A_{\xi_{a,b}^+},A_{\lambda};\overline{\mu_l})\\
		:=\mathlarger{\sum}_{k=0}^{l}
		(-1)^{\vert \mu_{k,l} \vert}(-q)^{-n(\mu_{k,l})-\vert \mu_{k,l} \vert+\langle (\xi_{a,b}^+)',\mu_{k,l}' \rangle} \prod_{j \geq 1} I_j(\mu_{k,l},\lambda).
		\end{array}
	\end{equation*}
\end{definition}

\begin{lemma}\label{lemma2.8}
For $\overline{\mu_l}=(\overline{\mu}_1,\dots,\overline{\mu}_{n-l}) \in \Lambda_{n-l}^{2+}$, we have that 
\begin{equation*}
	\sum_{i=0}^{s} d_{n,s,i}\alpha(A_{\xi_{0,i}^+},A_{\lambda};\overline{\mu_l})-\dfrac{(-1)^s}{(-q)^{ns}}\sum_{i=0}^{s}d_{n,s,i}\alpha(A_{\xi_{i,s-i}^+},A_{\lambda};\overline{\mu_l})=0.
\end{equation*}
\end{lemma}
\begin{proof}
Note that
\begin{equation*}
	\begin{array}{l}
		\vert \mu_{k,l} \vert=\vert \overline{\mu}_l \vert + k,\\
		n(\mu_{k,l})=n(\overline{\mu}_l)+(n-l)k+\dfrac{k(k-1)}{2}.
	\end{array}
\end{equation*}

Also, note that for $i \geq 2$, we have
\begin{equation*}
	(\mu_{0,l})'_i=(\mu_{1,l})'_i=\dots=(\mu_{l,l})'_i.
\end{equation*}

Recall that
\begin{equation*}\begin{array}{l}
	I_j(\mu_{k,l}, \lambda)\\
	=\mathlarger{\sum}_{i=(\mu_{k,l})_{j+1}'}^{\min((\tilde{\lambda})'_{j+1}, (\mu_{k,l})_j')}(-q)^{i(2(\tilde{\lambda})'_{j+1}+1-i)/2}\left[ \begin{array}{l}  
		(\tilde{\lambda})_{j+1}'-(\mu_{k,l})_{j+1}'\\
		(\tilde{\lambda})_{j+1}'-i
	\end{array}\middle] \middle[ \begin{array}{l} (\tilde{\lambda})_{j}'-i\\ (\tilde{\lambda})_{j}'-(\mu_{k,l})_{j}' \end{array}\right].
\end{array}
\end{equation*}

This depends only on $(\mu_{k,l})_{j}'$ and $(\mu_{k,l})_{j+1}'$. Therefore, for $j \geq 2$, $I_j(\mu_{k,l}, \lambda)$ are the same for all $k$.

We can therefore write that
	\begin{equation*}\begin{array}{l}
		\alpha(A_{\xi_{a,b}^+},A_{\lambda};\overline{\mu_l})\\
		=\mathlarger{\sum}_{k=0}^{l}
		(-1)^{\vert \mu_{k,l} \vert}(-q)^{-n(\mu_{k,l})-\vert \mu_{k,l} \vert+\langle (\xi_{a,b}^+)',\mu_{k,l}' \rangle} \prod_{j \geq 1} I_j(\mu_{k,l},\lambda)\\\\
		=
		(-1)^{\vert \overline{\mu_l} \vert}(-q)^{-n(\overline{\mu_l})-\vert \overline{\mu_l} \vert} \prod_{j \geq 2} I_j(\mu_{k,l},\lambda)\\
		\times \mathlarger{\sum}_{k=0}^{l}(-1)^{k}(-q)^{-(n-l)k-k(k+1)/2+\langle (\xi_{a,b}^+)',\mu_{k,l}' \rangle} I_1(\mu_{k,l},\lambda)
	\end{array}
\end{equation*}

Therefore, it suffices to prove that 
\begin{equation}\label{eq2.2.2}\begin{array}{l}
\mathlarger{\sum}_{i=0}^{s} d_{n,s,i} \mathlarger{\sum}_{k=0}^{l}(-1)^{k}(-q)^{-(n-l)k-k(k+1)/2+\langle (\xi_{0,i}^+)',\mu_{k,l}' \rangle} I_1(\mu_{k,l},\lambda)\\
-\dfrac{(-1)^s}{(-q)^{ns}}\mathlarger{\sum}_{i=0}^{s}d_{n,s,i} \mathlarger{\sum}_{k=0}^{l}(-1)^{k}(-q)^{-(n-l)k-k(k+1)/2+\langle (\xi_{i,s-i}^+)',\mu_{k,l}' \rangle} I_1(\mu_{k,l},\lambda)\\\\
=0.
\end{array}
\end{equation}

Note that
\begin{equation*}
	\begin{array}{l}
		(\xi_{a,b}^+)'_1=\xi'_1+a+b=n-s+a+b,\\
		(\xi_{a,b}^+)'_2=\xi'_2+a,\\
		(\xi_{a,b}^+)'_j=\xi'_j, \quad \forall j \geq 3,
	\end{array}
\end{equation*}
and 
\begin{equation*}
	\begin{array}{l}
		(\mu_{k,l})'_1=n-l+k,\\
		(\mu_{k,l})'_2=n-l.\\
	\end{array}
\end{equation*}
Also, note that for $j \geq 2$, $(\mu_{k,l})'_j$ are the same for all $k$

Therefore, we have that
\begin{equation*}\begin{array}{ll}
\mathlarger{\sum}_{i=0}^{s} d_{n,s,i}(-q)^{\langle (\xi_{0,i}^+)',\mu_{k,l}' \rangle}&=\mathlarger{\sum}_{i=0}^{s} d_{n,s,i}(-q)^{(n-s+i)(n-l+k)+(\xi_2')(n-l)+\sum_{j=3}^{\infty}\xi_j'(\mu_{k,l})'_j}\\\\
&=(-q)^{(n-s)(n-l+k)+(\xi_2')(n-l)+\sum_{j=3}^{\infty}\xi_j'(\mu_{k,l})'_j}\\
&\times \mathlarger{\prod}_{i=0}^{s-1}(1-(-q)^{-l+k+i}).
\end{array}
\end{equation*}

Also,
\begin{equation*}\begin{array}{ll}
		\mathlarger{\sum}_{i=0}^{s} d_{n,s,i}(-q)^{\langle (\xi_{i,s-i}^+)',\mu_{k,l}' \rangle}&=\mathlarger{\sum}_{i=0}^{s} d_{n,s,i}(-q)^{n(n-l+k)+(\xi_2'+i)(n-l)+\sum_{j=3}^{\infty}\xi_j'(\mu_{k,l})'_j}\\\\
		&=(-q)^{n(n-l+k)+(\xi_2')(n-l)+\sum_{j=3}^{\infty}\xi_j'(\mu_{k,l})'_j}\\
		&\times \mathlarger{\prod}_{i=0}^{s-1}(1-(-q)^{-l+i})
	\end{array}
\end{equation*}

These imply that the equation \eqref{eq2.2.2} is equivalent to
\begin{equation}\label{eq2.2.3}\begin{array}{l}
	\mathlarger{\sum}_{k=0}^{l}(-1)^{k}(-q)^{(n-s)k-(n-l)k-k(k+1)/2} \mathlarger{\prod}_{i=0}^{s-1}(1-(-q)^{-l+k+i}) I_1(\mu_{k,l},\lambda)\\
		-\dfrac{(-1)^s}{(-q)^{ns}}\mathlarger{\sum}_{k=0}^{l}(-1)^{k}(-q)^{s(n-l+k)+(n-s)k-(n-l)k-k(k+1)/2}\mathlarger{\prod}_{i=0}^{s-1}(1-(-q)^{-l+i}) I_1(\mu_{k,l},\lambda)\\\\
		=0.\\\\
		
		\Longleftrightarrow\\
		\mathlarger{\sum}_{k=0}^{l}(-1)^{k}(-q)^{(l-s)k-k(k+1)/2}I_1(\mu_{k,l},\lambda)\\
		\times \lbrace \mathlarger{\prod}_{i=0}^{s-1}(1-(-q)^{-l+k+i}) -(-1)^s(-q)^{s(-l+k)}\mathlarger{\prod}_{i=0}^{s-1}(1-(-q)^{-l+i})\rbrace =0.
	\end{array}
\end{equation}

We have that
\begin{equation*}\begin{array}{l}
		I_1(\mu_{k,l}, \lambda)\\
		=\mathlarger{\sum}_{i=(\mu_{k,l})_{2}'}^{\min((\tilde{\lambda})'_{2}, (\mu_{k,l})_1')}(-q)^{i(2(\tilde{\lambda})'_{2}+1-i)/2}\left[ \begin{array}{l}  
			(\tilde{\lambda})_{2}'-(\mu_{k,l})_{2}'\\
			(\tilde{\lambda})_{2}'-i
		\end{array}\middle] \middle[ \begin{array}{l} (\tilde{\lambda})_{1}'-i\\ (\tilde{\lambda})_{1}'-(\mu_{k,l})_{1}' \end{array}\right].
	\end{array}
\end{equation*}
Since we assume that $\lambda \geq (1,1,\dots,1)$, we have that
\begin{equation*}
	\begin{array}{l}
		(\tilde{\lambda})_2'=(\tilde{\lambda})_1'=n.
		\end{array}
\end{equation*}
Therefore,
\begin{equation*}\begin{array}{ll}
		I_1(\mu_{k,l}, \lambda)
		&=\mathlarger{\sum}_{i=n-l}^{n-l+k}(-q)^{i(2n+1-i)/2}\left[ \begin{array}{c}  
			l\\
			n-i
		\end{array}\middle] \middle[ \begin{array}{c} n-i\\ l-k \end{array}\right]\\\\
	
	&=\mathlarger{\sum}_{j=0}^{k}(-q)^{(j+n-l)(n+l+1-j)/2}\left[ \begin{array}{c}  
		l\\
		l-j
	\end{array}\middle] \middle[ \begin{array}{c} l-j\\ l-k \end{array}\right]\\\\
&=(-q)^{(n-l)(n+l+1)/2}\mathlarger{\sum}_{j=0}^{k}(-q)^{(2l+1-j)j/2}\left[ \begin{array}{c}  
	l\\
	l-j
\end{array}\middle] \middle[ \begin{array}{c} l-j\\ l-k \end{array}\right]\\\\
	\end{array}
\end{equation*}

This implies that the equation \eqref{eq2.2.3} is equivalent to

\begin{equation}\label{eq2.2.4}\begin{array}{l}
		\mathlarger{\sum}_{k=0}^{l}(-1)^{k}(-q)^{(l-s)k-k(k+1)/2}\mathlarger{\sum}_{j=0}^{k}(-q)^{(2l+1-j)j/2}\left[ \begin{array}{c}  
			l\\
			l-j
		\end{array}\middle] \middle[ \begin{array}{c} l-j\\ l-k \end{array}\right]\\
		\times \lbrace \mathlarger{\prod}_{i=0}^{s-1}(1-(-q)^{-l+k+i}) -(-1)^s(-q)^{s(-l+k)}\mathlarger{\prod}_{i=0}^{s-1}(1-(-q)^{-l+i})\rbrace =0.
	\end{array}
\end{equation}

This equality will be proved separately in Lemma \ref{lemma2.11}.
\end{proof}

\begin{lemma}\label{lemma2.10}
	\begin{equation*}
		\mathlarger{\sum}_{j=0}^{k}(-q)^{(2l+1-j)j/2}\left[ \begin{array}{c}  
			l\\
			l-j
		\end{array}\middle] \middle[ \begin{array}{c} l-j\\ l-k \end{array}\middle]=\middle[ \begin{array}{c} l\\ k \end{array}\right]\mathlarger{\prod}_{j=0}^{k-1}(1+(-q)^{l-j}).
	\end{equation*}
	
	\end{lemma}
\begin{proof}
	First, note that
	\begin{equation*}\begin{array}{ll}
			\left[ \begin{array}{c}  
			l\\
			l-j
		\end{array}\middle] \middle[ \begin{array}{c} l-j\\ l-k \end{array}\right]&=\dfrac{\prod_{i=1}^{l}(1-(-q)^{-i})}{\prod_{i=1}^{l-j}(1-(-q)^{-i})\prod_{i=1}^{j}(1-(-q)^{-i})}\\\\
	&\times \dfrac{\prod_{i=1}^{l-j}(1-(-q)^{-i})}{\prod_{i=1}^{k-j}(1-(-q)^{-i})\prod_{i=1}^{l-k}(1-(-q)^{-i})}\qquad\qquad\qquad\qquad\qquad
	\end{array}
	\end{equation*}
		
		\begin{equation*}\begin{array}{ll}
		\quad\qquad\qquad\qquad\qquad&=\dfrac{\prod_{i=1}^{l}(1-(-q)^{-i})}{\prod_{i=1}^{l-k}(1-(-q)^{-i})\prod_{i=1}^{j}(1-(-q)^{-i})\prod_{i=1}^{k-j}(1-(-q)^{-i})}\\
		\end{array}
\end{equation*}

\begin{equation*}\begin{array}{ll}
	&=\dfrac{\prod_{i=1}^{l}(1-(-q)^{-i})}{\prod_{i=1}^{k}(1-(-q)^{-i})\prod_{i=1}^{l-k}(1-(-q)^{-i})}\\
	&\times \dfrac{\prod_{i=1}^{k}(1-(-q)^{-i})}{\prod_{i=1}^{j}(1-(-q)^{-i})\prod_{i=1}^{k-j}(1-(-q)^{-i})}
\end{array}
\end{equation*}

\begin{equation*}\begin{array}{ll}
	\qquad&=\left[ \begin{array}{c}  
		l\\
		k
	\end{array}\middle] \middle[ \begin{array}{c} k\\ j \end{array}\right]\qquad\qquad\qquad\qquad\qquad\qquad
	\end{array}
	\end{equation*}

Therefore, it suffices to show that
	\begin{equation}\label{eq2.2.5}
	\mathlarger{\sum}_{j=0}^{k}(-q)^{(2l+1-j)j/2}\left[ \begin{array}{c}  
		k\\
		j\end{array}\right]=\mathlarger{\prod}_{j=0}^{k-1}(1+(-q)^{l-j}).
\end{equation}

We prove \eqref{eq2.2.5} by induction on $k$.
When $k=1$, we have that
\begin{equation*}
	\mathlarger{\sum}_{j=0}^{1}(-q)^{(2l+1-j)j/2}\left[ \begin{array}{c}  
		1\\
		j\end{array}\right]=1+(-q)^{l}.
\end{equation*}
Therefore, \eqref{eq2.2.5} is true for $k=1$.
Now, assume that \eqref{eq2.2.5} is true for $k$.

Note that
\begin{equation*}\begin{array}{ll}
	\left[ \begin{array}{c}  
		k+1\\
		j\end{array}\right] &=\dfrac{\prod_{i=1}^{k+1}(1-(-q)^{-i})}{\prod_{i=1}^{j}(1-(-q)^{-i})\prod_{i=1}^{k+1-j}(1-(-q)^{-i})}\\\\
	
	&=\dfrac{(1-(-q)^{-k-1})}{(1-(-q)^{-k-1+j})} \left[ \begin{array}{c}  
		k\\
	j\end{array}\right]
	\end{array}
\end{equation*}

Therefore,
\begin{equation*}\begin{array}{l}
		\left[ \begin{array}{c}  
			k+1\\
			j\end{array}\right]-\left[ \begin{array}{c}  
			k\\
			j\end{array}\right] =\dfrac{((-q)^{-k-1+j}-(-q)^{-k-1})}{(1-(-q)^{-k-1+j})} \left[ \begin{array}{c}  
			k\\
			j\end{array}\right]\\
		=(-q)^{-k-1+j}\dfrac{(1-(-q)^{-j})}{(1-(-q)^{-k-1+j})} \dfrac{(1-(-q)^{-k+j-1})}{(1-(-q)^{-j})}\left[ \begin{array}{c}  
			k\\
			j-1\end{array}\right]\\
		=(-q)^{-k-1+j}\left[ \begin{array}{c}  
			k\\
			j-1\end{array}\right]
	\end{array}
\end{equation*}

This implies that
\begin{equation*}
	\begin{array}{l}
	\mathlarger{\sum}_{j=0}^{k+1}(-q)^{(2l+1-j)j/2}\left[ \begin{array}{c}  
		k+1\\
		j\end{array}\middle]-	\mathlarger{\sum}_{j=0}^{k}(-q)^{(2l+1-j)j/2}\middle[ \begin{array}{c}  
			k\\
			j\end{array}\right]\\
		=	\mathlarger{\sum}_{j=0}^{k}(-q)^{(2l+1-j)j/2}\lbrace \left[ \begin{array}{c}  
			k+1\\
			j\end{array}\middle]-\middle[ \begin{array}{c}  
			k\\
			j\end{array}\right] \rbrace+(-q)^{(2l-k)(k+1)/2}.
	\end{array}
\end{equation*}
Since $\left[ \begin{array}{c}  
	k+1\\
	0\end{array}\middle]=\middle[ \begin{array}{c}  
	k\\
	0\end{array}\right]=1$, we have that

\begin{equation*}
	\begin{array}{l}
		\mathlarger{\sum}_{j=0}^{k}(-q)^{(2l+1-j)j/2}\lbrace \left[ \begin{array}{c}  
			k+1\\
			j\end{array}\middle]-\middle[ \begin{array}{c}  
			k\\
			j\end{array}\right] \rbrace+(-q)^{(2l-k)(k+1)/2}\\
			=\mathlarger{\sum}_{j=1}^{k}(-q)^{(2l+1-j)j/2}\lbrace \left[ \begin{array}{c}  
			k+1\\
			j\end{array}\middle]-\middle[ \begin{array}{c}  
			k\\
			j\end{array}\right] \rbrace+(-q)^{(2l-k)(k+1)/2}\\
			=\mathlarger{\sum}_{j=1}^{k}(-q)^{(2l+1-j)j/2}(-q)^{-k-1+j}\left[ \begin{array}{c}  
				k\\
				j-1\end{array}\right]+(-q)^{(2l-k)(k+1)/2}\\
			=\mathlarger{\sum}_{j=0}^{k-1}(-q)^{(2l-j)(j+1)/2}(-q)^{-k+j}\left[ \begin{array}{c}  
				k\\
				j\end{array}\right]+(-q)^{(2l-k)(k+1)/2}\\
			=\mathlarger{\sum}_{j=0}^{k}(-q)^{(2l-j)(j+1)/2}(-q)^{-k+j}\left[ \begin{array}{c}  
				k\\
				j\end{array}\right]\\
			=(-q)^{l-k}\mathlarger{\sum}_{j=0}^{k}(-q)^{(2l+1-j)j/2}\left[ \begin{array}{c}  
				k\\
				j\end{array}\right]\\
	\end{array}
\end{equation*}

Therefore, we have that
\begin{equation*}
	\begin{array}{l}
		\mathlarger{\sum}_{j=0}^{k+1}(-q)^{(2l+1-j)j/2}\left[ \begin{array}{c}  
			k+1\\
			j\end{array}\middle]=	(1+(-q)^{l-k})\mathlarger{\sum}_{j=0}^{k}(-q)^{(2l+1-j)j/2}\middle[ \begin{array}{c}  
			k\\
			j\end{array}\right].
	\end{array}
\end{equation*}

By our inductive hypothesis, we have
\begin{equation*}
	\begin{array}{l}
		\mathlarger{\sum}_{j=0}^{k+1}(-q)^{(2l+1-j)j/2}\left[ \begin{array}{c}  
			k+1\\
			j\end{array}\right]=	(1+(-q)^{l-k})\mathlarger{\prod}_{j=0}^{k-1}(1+(-q)^{l-j}).
	\end{array}
\end{equation*}

This finishes the proof of \eqref{eq2.2.5} and hence the proof of the lemma.

\end{proof}

\begin{lemma}\label{lemma2.11}
	We have that
	\begin{equation*}\begin{array}{l}
			\mathlarger{\sum}_{k=0}^{l}(-1)^{k}(-q)^{(l-s)k-k(k+1)/2}\mathlarger{\sum}_{j=0}^{k}(-q)^{(2l+1-j)j/2}\left[ \begin{array}{c}  
				l\\
				l-j
			\end{array}\middle] \middle[ \begin{array}{c} l-j\\ l-k \end{array}\right]\\
			\times \lbrace \mathlarger{\prod}_{i=0}^{s-1}(1-(-q)^{-l+k+i}) -(-1)^s(-q)^{s(-l+k)}\mathlarger{\prod}_{i=0}^{s-1}(1-(-q)^{-l+i})\rbrace =0.
		\end{array}
	\end{equation*}
\end{lemma}
\begin{proof}	By Lemma \ref{lemma2.10}, we need to show that
	\begin{equation*}\begin{array}{l}
			\mathlarger{\sum}_{k=0}^{l}(-1)^{k}(-q)^{(l-s)k-k(k+1)/2}\left[ \begin{array}{c} l\\ k \end{array}\right]\mathlarger{\prod}_{j=0}^{k-1}(1+(-q)^{l-j}).\\
			\times \lbrace \mathlarger{\prod}_{i=0}^{s-1}(1-(-q)^{-l+k+i}) -(-1)^s(-q)^{s(-l+k)}\mathlarger{\prod}_{i=0}^{s-1}(1-(-q)^{-l+i})\rbrace =0.
		\end{array}
	\end{equation*}

Let
\begin{equation*}\begin{array}{l}
	\alpha_{sk}:=(-1)^{k}(-q)^{(l-s)k-k(k+1)/2}\left[ \begin{array}{c} l\\ k \end{array}\right]\mathlarger{\prod}_{j=0}^{k-1}(1+(-q)^{l-j}).\\
	\times \lbrace \mathlarger{\prod}_{i=0}^{s-1}(1-(-q)^{-l+k+i}) -(-1)^s(-q)^{s(-l+k)}\mathlarger{\prod}_{i=0}^{s-1}(1-(-q)^{-l+i})\rbrace
	\end{array}
\end{equation*}

Since $\left[ \begin{array}{c} l\\ k \end{array}\right]=\dfrac{\prod_{i=l-k+1}^l(1-(-q)^{-i})}{\prod_{i=1}^k(1-(-q)^{-i})}$ and $(1+(-q)^{l-j})=(-q)^{l-j}(1+(-q)^{-l+j})$ we have
\begin{equation*}\begin{array}{ll}
	\alpha_{sk}&=\dfrac{\prod_{i=l-k+1}^l(1-(-q)^{-2i})}{\prod_{i=1}^k(1-(-q)^{-i})}\\
	&\times \mathlarger{\mathlarger{\lbrace}} (-1)^k(-q)^{2lk-k^2-sk}\mathlarger{\prod}_{i=0}^{s-1}(1-(-q)^{-l+k+i})\\ &-(-1)^{s+k}(-q)^{2lk-k^2-sl}\mathlarger{\prod}_{i=0}^{s-1}(1-(-q)^{-l+i})\mathlarger{\mathlarger{\rbrace}}.
	\end{array}
\end{equation*}

Now, we want to prove that $\sum_{k=0}^{l}\alpha_{sk}=0$.

Assume that $l < s$. Then, $\alpha_{sk}$ is zero for all $k$, and hence $\sum_{k=0}^{l}\alpha_{sk}=0$. 

Now, assume that $l \geq s$.

For $0 \leq t \leq l$, we write $B_{t,1}$ and $B_{t,2}$ for
\begin{equation*}
	\begin{array}{l}
		B_{t,1}:=(-1)^t(-q)^{2lt-t^2-st}\mathlarger{\prod}_{i=0}^{s-1}(1-(-q)^{-l+t+i});\\
		B_{t,2}:=-(-1)^{s+t}(-q)^{2lt-t^2-sl}\mathlarger{\prod}_{i=0}^{s-1}(1-(-q)^{-l+i}).
	\end{array}
\end{equation*}
Also, for $0 \leq t \leq l$, $0 \leq k \leq s-1$, we write $L_{t,k}$ and $R_{t,k}$ for
\begin{equation*}
	\begin{array}{ll}
		L_{t,k}:=&(-1)^{t-k}(-q)^{l(2t-k)-t(t+s-k)}\prod_{i=1}^{k}(1-(-q)^{-l+s-i})\\
		&\quad\quad\quad\quad\quad\quad\quad\quad\quad\quad\quad\times \prod_{i=1}^{s-k-1}(1-(-q)^{-l+t+i}),\\\\
		
		R_{t,k}:=&(-1)^{t-k}(-q)^{l(2t-k-1)-t(t+s-k-1)}\prod_{i=1}^{k}(1-(-q)^{-l+s-i})\\
		&\quad\quad\quad\quad\quad\quad\quad\quad\quad\quad\quad\quad\quad\times \prod_{i=0}^{s-k-1}(1-(-q)^{-l+t+i}).\\
	\end{array}
\end{equation*}
Here, we choose the following convention. The term $\prod_{i=1}^{k}(1-(-q)^{-l+s-i})$ is $1$ if $k=0$, and the term $\prod_{i=1}^{s-k-1}(1-(-q)^{-l+t+i})$ is $1$ if $k=s-1$.

\begin{claim}\label{claim2.12} For $0 \leq t \leq l$, and $0 \leq k \leq s-1$ we have the followings.
	\begin{enumerate}
		\item $\alpha_{s0}=(1-(-q)^{-2l}) \sum_{k=0}^{s-1}L_{0,k}$.\\
		
		\item $L_{t,s-1}(1-(-q)^{-t-1})+B_{t+1,2}=R_{t+1,s-1}$.\\
		
		\item For $2 \leq r \leq s$, we have
		\begin{equation*}
			L_{t,s-r}(1-(-q)^{-t-1})+R_{t+1,s-r+1}=L_{t+1,s-r+1}(1-(-q)^{-2l+2t+2)}+R_{t+1,s-r}.
		\end{equation*}
	\quad
	\item $R_{t+1,0}+B_{t+1,1}=L_{t+1,0}(1-(-q)^{-2l+2t+2})$.\\
	
	\item $\sum_{k=0}^{s-1}L_{t,k}(1-(-q)^{-t-1})+B_{t+1,1}+B_{t+1,2}=\sum_{k=0}^{s-1}L_{t+1,k}(1-(-q)^{-2l+2t+2})$
	\end{enumerate}
\end{claim}
\begin{proof}[Proof of Claim \ref{claim2.12}]
	\begin{enumerate}
		\item Note that
		\begin{equation*}\begin{array}{ll}
				\alpha_{s0}&=(1-(-1)^s(-q)^{-sl})\mathlarger{\prod}_{i=0}^{s-1}(1-(-q)^{-l+i})\\
				&=(1+(-q)^{-l})\sum_{k=0}^{s-1}(-1)^{k}(-q)^{-kl}\mathlarger{\prod}_{i=0}^{s-1}(1-(-q)^{-l+i})\\
				&=(1-(-q)^{-2l})\sum_{k=0}^{s-1}L_{0,k}.
			\end{array}
		\end{equation*}
	This finishes the proof of (1).\\
	
	\item We have that
	\begin{equation*}\begin{array}{l}
		L_{t,s-1}(1-(-q)^{-t-1})+B_{t+1,2}\\
		=(-1)^{t-s+1}(-q)^{l(2t-s+1)-t(t+1)}\prod_{i=1}^{s-1}(1-(-q)^{-l+s-i}) (1-(-q)^{-t-1})\\
		-(-1)^{s+t+1}(-q)^{l(2t-s+2)-(t+1)^2}\mathlarger{\prod}_{i=0}^{s-1}(1-(-q)^{-l+i}).\\\\
		=(-1)^{t-s+2}(-q)^{l(2t-s+2)-(t+1)^2}\prod_{i=1}^{s-1}(1-(-q)^{-l+s-i})\\
		\times \lbrace 1-(-q)^l-(-q)^{-l+t+1}(1-(-q)^{-t-1})\rbrace\\\\
		=(-1)^{t-s+2}(-q)^{l(2t-s+2)-(t+1)^2}\prod_{i=1}^{s-1}(1-(-q)^{-l+s-i})(1-(-q)^{-l+t+1})\\
		=R_{t+1,s-1}.
		\end{array}
	\end{equation*}
This finishes the proof of (2).\\

\item We have that
	\begin{equation*}\begin{array}{l}
		L_{t,s-r}(1-(-q)^{-t-1})+R_{t+1,s-r+1}\\
		=(-1)^{t-s+r}(-q)^{l(2t-s+r)-t(t+r)}\prod_{i=1}^{s-r}(1-(-q)^{-l+s-i}) \\ \quad\quad\quad\quad\quad\quad\quad\quad\quad\quad\quad\quad\quad\quad\prod_{i=1}^{r-1}(1-(-q)^{-l+t+i})(1-(-q)^{-t-1})\\
		
		+(-1)^{t-s+r}(-q)^{l(2t-s+r)-(t+1)(t+r-1)}\prod_{i=1}^{s-r+1}(1-(-q)^{-l+s-i}) \\ \quad\quad\quad\quad\quad\quad\quad\quad\quad\quad\quad\quad\quad\quad\quad\quad\prod_{i=0}^{r-2}(1-(-q)^{-l+t+1+i})\\\\
		
		=(-1)^{t-s+r}\prod_{i=1}^{s-r}(1-(-q)^{-l+s-i})\prod_{i=1}^{r-1}(1-(-q)^{-l+t+i}) \\
		
		(-q)^{l(2t-s+r)-(t+1)(t+r-1)}\lbrace (-q)^{r-1}(1-(-q)^{-t-1})+(1-(-q)^{-l+r-1})\rbrace\\\\
		
		=-(-q)^{-t+r-2}(1+(-q)^{-l+t+1})(-1)^{t-s+r}(-q)^{l(2t-s+r)-(t+1)(t+r-1)}\\
		(-q)^{l-r+1}(-q)^{-l+r-1}\prod_{i=1}^{s-r}(1-(-q)^{-l+s-i})\prod_{i=1}^{r-1}(1-(-q)^{-l+t+i})\\\\
		
		+((-q)^{r-1}+1)(-1)^{t-s+r}(-q)^{l(2t-s+r)-(t+1)(t+r-1)}\\
		\quad\quad\quad\quad\prod_{i=1}^{s-r}(1-(-q)^{-l+s-i})\prod_{i=1}^{r-1}(1-(-q)^{-l+t+i})\\
		\end{array}\end{equation*}
	\begin{equation*}\begin{array}{l}
		=(1-(-q)^{-2l+2t+2})(-1)^{t-s+r}(-q)^{l(2t-s+r+1)-(t+1)(t+r)}\\
	\quad\quad\quad	\prod_{i=1}^{s-r+1}(1-(-q)^{-l+s-i})\prod_{i=2}^{r-1}(1-(-q)^{-l+t+i})\\\\
	
	-(1-(-q)^{-2l+2t+2})(-1)^{t-s+r}(-q)^{l(2t-s+r+1)-(t+1)(t+r)}\\
	\quad\quad\quad	\prod_{i=1}^{s-r}(1-(-q)^{-l+s-i})\prod_{i=2}^{r-1}(1-(-q)^{-l+t+i})\\
	+((-q)^{r-1}+1)(-1)^{t-s+r}(-q)^{l(2t-s+r)-(t+1)(t+r-1)}\\
	\quad\quad\quad\quad\prod_{i=1}^{s-r}(1-(-q)^{-l+s-i})\prod_{i=1}^{r-1}(1-(-q)^{-l+t+i})\\\\
	
	=L_{t+1,s-r+1}(1-(-q)^{-2l+2t+2})\\\\
	
	+(-1)^{t-s+r+1}(-q)^{l(2t-s+r+1)-(t+1)(t+r)}\\
	\quad\quad\quad\quad\prod_{i=1}^{s-r}(1-(-q)^{-l+s-i})\prod_{i=1}^{r-1}(1-(-q)^{-l+t+i})\\
	\quad\quad\quad\quad\quad\quad\quad\quad \times 
	\lbrace 1+(-q)^{-l+t+1}-((-q)^{r-1}+1)(-q)^{-l+t+1}\rbrace\quad\quad\quad\\\\
	=L_{t+1,s-r+1}(1-(-q)^{-2l+2t+2})+R_{t+1,s-r}.
	\end{array}
\end{equation*}
This finishes the proof of (3).\\

\item We have that
\begin{equation*}
	\begin{array}{l}
		R_{t+1,0}+B_{t+1,1}\\
		=(-1)^{t+1}(-q)^{l(2t+1)-(t+1)(t+s)}\prod_{i=0}^{s-1}(1-(-q)^{-l+t+1+i})\\
		+(-1)^{t+1}(-q)^{l(2t+2)-(t+1)(t+s+1)}\prod_{i=0}^{s-1}(1-(-q)^{-l+t+1+i})\\\\
		=(-1)^{t+1}(-q)^{l(2t+2)-(t+1)(t+s+1)}\prod_{i=0}^{s-1}(1-(-q)^{-l+t+1+i})(1+(-q)^{-l+t+1})\\
		=L_{t+1,0}(1-(-q)^{-2l+2t+2}).
		\end{array}
\end{equation*}
This finishes the proof of (4).\\
\item This is obvious from (1), (2), (3), and (4).
	\end{enumerate}
\end{proof}

Now, by Claim \ref{claim2.12} (1), we have that $\alpha_{s0}=(1-(-q)^{-2l}) \sum_{k=0}^{s-1}L_{0,k}$. Also, by definition, we have that
\begin{equation*}
	\alpha_{st}=\dfrac{\prod_{i=0}^{t-1}(1-(-q)^{-2l+2i})}{\prod_{i=1}^{t}(1-(-q)^{-i})}(B_{t,1}+B_{t+2}).
\end{equation*}

Therefore, by Claim \ref{claim2.12} (5) we have
\begin{equation*}\begin{array}{ll}
	\alpha_{s0}+\alpha_{s1}&=\dfrac{(1-(-q)^{-2l})}{(1-(-q)^{-1})}\lbrace\sum_{k=0}^{s-1}L_{0,k}(1-(-q)^{-1})+B_{1,1}+B_{1,2}\rbrace\\\\
		&=\dfrac{(1-(-q)^{-2l})(1-(-q)^{-2l+2})}{(1-(-q)^{-1})}\lbrace\sum_{k=0}^{s-1}L_{1,k}\rbrace
		\end{array}
\end{equation*}

Now, we claim that for any $t\leq l$,
\begin{equation}\label{eq2.2.6}
	\sum_{i=0}^{t}\alpha_{si}=\dfrac{\prod_{i=0}^{t}(1-(-q)^{-2l+2i})}{\prod_{i=1}^{t}(1-(-q)^{-i})}\sum_{k=0}^{s-1}L_{t,k}.
\end{equation}

We prove this by induction on $t$. Assume that \eqref{eq2.2.6} is true for $t$, then we have

\begin{equation*}
	\sum_{i=0}^{t+1}\alpha_{si}=\dfrac{\prod_{i=0}^{t}(1-(-q)^{-2l+2i})}{\prod_{i=1}^{t+1}(1-(-q)^{-i})}\lbrace\sum_{k=0}^{s-1}L_{t,k}(1-(-q)^{-t-1})+B_{t+1,1}+B_{t+1,2}\rbrace.
\end{equation*}
By Claim \ref{claim2.12} (5), we have
\begin{equation*}
	\sum_{i=0}^{t+1}\alpha_{si}=\dfrac{\prod_{i=0}^{t+1}(1-(-q)^{-2l+2i})}{\prod_{i=1}^{t+1}(1-(-q)^{-i})}\lbrace\sum_{k=0}^{s-1}L_{t+1,k}\rbrace.
\end{equation*}
This finishes the proof of\eqref{eq2.2.6}.

Therefore, we see that \begin{equation*}\sum_{i=0}^{l}\alpha_{si}=\dfrac{\prod_{i=0}^{l}(1-(-q)^{-2l+2i})}{\prod_{i=1}^{l}(1-(-q)^{-i})}\lbrace\sum_{k=0}^{s-1}L_{l,k}\rbrace=0,
	\end{equation*}
and this completes the proof of the lemma.
\end{proof}
\quad\\
\begin{proof}[Proof of Theorem 2.5]
	
	This follows from \eqref{eq2.2.1} and Lemma \ref{lemma2.8}.
\end{proof}
\bigskip

\subsection{Normalized representation densities}
When we apply Theorem \ref{theorem2.5} to compute representation densities, it is better to use a normalized version of the theorem. In this section, we consider a normalized version of Theorem \ref{theorem2.5} and compute some representation densities.

\begin{definition}
	For $A,B \in X_n(O_E)$, we define $A(B)$ by
	\begin{equation*}
		A(B):=\dfrac{\alpha(A,B)}{\alpha(A,A)}.
		\end{equation*}
\end{definition}

These normalized representation densities have the following reduction properties.

\begin{proposition}\label{proposition2.14}
	For $\xi, \lambda \in \Lambda_{n-1}^+$, we consider $(\xi,0), (\lambda,0)$ in $\Lambda_n^+$. Then, we have
	\begin{equation*}
		A_{\xi,0}(A_{\lambda,0})=A_{\xi}(A_{\lambda}).
		\end{equation*}
\end{proposition}

\begin{proof}
	By \cite[Corollary 9.12]{KR2}, we have that
	\begin{equation*}\begin{array}{l}
		\alpha(A_{\xi,0},A_{\lambda,0})=\alpha(A_{\xi,0},A_0)\alpha(A_{\xi},A_{\lambda}),\\
		\alpha(A_{\xi,0},A_{\xi,0})=\alpha(A_{\xi,0},A_0)\alpha(A_{\xi},A_{\xi}).
		\end{array}
	\end{equation*}

Therefore, we have
	\begin{equation*}
	A_{\xi,0}(A_{\lambda,0})=A_{\xi}(A_{\lambda}).
\end{equation*}
\end{proof}

\begin{proposition}\label{proposition2.15}
	For $\xi, \lambda$ in $\Lambda_n^+$, such that $\xi, \lambda \geq (1,1\dots,1)$, we consider
	\begin{equation*}\begin{array}{l}
		\xi-1:=(\xi_1-1,\dots,\xi_{n}-1) \in \Lambda_n^+\\
		\lambda-1:=(\lambda_1-1,\dots,\lambda_n-1) \in \Lambda_n^+.
		\end{array}
	\end{equation*}
Then, we have
\begin{equation*}
	A_{\xi}(A_{\lambda})=A_{\xi-1}(A_{\lambda-1}).
\end{equation*}
\end{proposition}
\begin{proof}
	This follows directly from the fact that
	\begin{equation*}\begin{array}{l}
		\alpha(A_{\xi},A_{\xi})=q^{n^2}\alpha(A_{\xi-1},A_{\xi-1}),\\
		\alpha(A_{\xi},A_{\lambda})=q^{n^2}\alpha(A_{\xi-1},A_{\lambda-1})
		\end{array}
		\end{equation*}
	\end{proof}

\begin{example}[$n=1$] By Theorem \ref{theorem2.5}, for $B \in X_1(O_E)$ such that $\pi^{-1}B$ is integral, we have that
	\begin{equation*}
		\lbrace \alpha(A_0,B)+\dfrac{1}{q}\alpha(A_1,B) \rbrace+\dfrac{1}{(-q)}\lbrace\alpha(A_1,B)+\dfrac{1}{q}\alpha(A_2,B)\rbrace=0.
	\end{equation*}
Therefore, $\alpha(A_0,B)=\dfrac{1}{q^2}\alpha(A_2,B)$.

Since $\alpha(A_0,A_0)=(1+\dfrac{1}{q})$ and $\alpha(A_2,A_2)=q^2(1+\dfrac{1}{q})$, we have
\begin{equation*}
	A_0(B)=A_2(B).
	\end{equation*}

For example, for $B=A_{6}=(\pi^{6})$, we have
\begin{equation*}
	A_0(A_{6})=A_2(A_{6})=A_0(A_4)=A_2(A_4)=A_0(A_2)=A_2(A_2)=1.
\end{equation*}

In general, $A_0(B)$ is $1$ for any $B$ such that $\val \det B \equiv 0 (\Mod 2)$.
\end{example}

\begin{corollary}[$n=2$]\label{corollary2.17}
	Assume that $\lambda \geq 3$. For any $B \in X_2(O_E)$ such that $\pi^{-1}B$ is integral, the following equalities hold.
	\begin{enumerate}
		\item $A_{00}(B)=(q+1)A_{11}(B)-qA_{22}(B)$.\\
		
		\item $A_{10}(B)=A_{21}(B)$.\\
		
		\item $A_{20}(B)=q(q-1)A_{22}(B)$.\\
		
		\item $A_{\lambda0}(B)=q^2A_{\lambda2}(B)$.
	\end{enumerate}
\end{corollary}
\begin{proof}
	By Theorem \ref{theorem2.5} for $n=2$ and $s=2$, we have that
	\begin{equation}\label{eq2.3.1}\begin{array}{l}
		\lbrace \alpha(A_{00},B)+(\dfrac{1}{q}-\dfrac{1}{q^2})\alpha(A_{10},B)-\dfrac{1}{q^3}\alpha(A_{11},B)\rbrace\\\\
		-\dfrac{1}{q^4}\lbrace \alpha(A_{11},B)+(\dfrac{1}{q}-\dfrac{1}{q^2})\alpha(A_{21},B)-\dfrac{1}{q^3}\alpha(A_{22},B)\rbrace=0
		\end{array}
	\end{equation}

Also, by \cite[Lemma 4.6]{Cho3}, we can compute that
\begin{equation*}
	\alpha(A_{00},A_{00})=(1+\dfrac{1}{q})(1-\dfrac{1}{q^2}),
\end{equation*}
\begin{equation*}
\alpha(A_{10},A_{10})=q(1+\dfrac{1}{q})^2,
\end{equation*}
\begin{equation*}
	\alpha(A_{11},A_{11})=q^4(1+\dfrac{1}{q})(1-\dfrac{1}{q^2}),
\end{equation*}
\begin{equation*}
	\alpha(A_{21},A_{21})=q^5(1+\dfrac{1}{q})^2,
\end{equation*}
\begin{equation*}
	\alpha(A_{22},A_{22})=q^8(1+\dfrac{1}{q})(1-\dfrac{1}{q^2}),
\end{equation*}

Therefore, \eqref{eq2.3.1} is
\begin{equation*}
A_{00}(B)+A_{10}(B)-(q+1)A_{11}(B)-A_{21}(B)+qA_{22}(B)=0.
\end{equation*}

Note that $A(B)=0$ if $\val \det A \not\equiv \val \det B (\Mod 2)$. Therefore, we have
\begin{equation*}
	A_{00}(B)=(q+1)A_{11}(B)-qA_{22}(B),
\end{equation*}
and
\begin{equation*}
	A_{10}(B)=A_{21}(B).
\end{equation*}
This finishes the proof of (1) and (2).

Now, for $\lambda \geq 2$, by Theorem \ref{theorem2.5} for $n=2$, $s=1$, we have
\begin{equation*}
	\lbrace\alpha(A_{\lambda0},B)-\dfrac{1}{q^2}\alpha(A_{\lambda1},B)\rbrace+\dfrac{1}{q^2}\lbrace \alpha(A_{\lambda1},B)-\dfrac{1}{q^2}\alpha(A_{\lambda2},B)\rbrace=0.
\end{equation*}
Therefore,
\begin{equation}\label{eq2.3.2}
	\alpha(A_{\lambda0},B)=\dfrac{1}{q^4}\alpha(A_{\lambda2},B).
\end{equation}

Now, for $\lambda \geq 3$, we have
\begin{equation*}\begin{array}{l}
	\alpha(A_{\lambda0},A_{\lambda,0})=(1+\dfrac{1}{q})\alpha(A_{\lambda},A_{\lambda})=(1+\dfrac{1}{q})q^2\alpha(A_{\lambda-2},A_{\lambda-2}),\\\\
	\alpha(A_{\lambda2},A_{\lambda,2})=q^8(1+\dfrac{1}{q})\alpha(A_{\lambda-2},A_{\lambda-2}).
	\end{array}
\end{equation*}
Also, for $\lambda=2$, we have
\begin{equation*}
	\begin{array}{l}
		\alpha(A_{20},A_{20})=q^2(1+\dfrac{1}{q})^2,\\\\
		\alpha(A_{22},A_{22})=q^8(1+\dfrac{1}{q})(1-\dfrac{1}{q^2}).
	\end{array}
\end{equation*}

These imply that \eqref{eq2.3.2} is
\begin{equation*}\begin{array}{l}
	A_{\lambda0}(B)=q^2A_{\lambda2}(B), \text{ if }\lambda \geq 3;\\
	
	A_{20}(B)=q(q-1)A_{22}(B).
	\end{array}
\end{equation*}

This finishes the proof of (3) and (4).
\end{proof}

\begin{remark}
	When $n=2$, Corollary \ref{corollary2.17} gives an efficient way to compute representation densities. For example, consider $A_{60}(A_{86})$. If we use Proposition \ref{proposition2.2}, we need several pages to compute this. But, by Corollary \ref{corollary2.17}, we can compute that
	\begin{equation*}\begin{array}{l}
		A_{60}(A_{86})=q^2A_{62}(A_{86})=q^2A_{40}(A_{64})=q^4A_{42}(A_{64})=q^4A_{20}(A_{42})\\\\
		=q^5(q-1)A_{22}(A_{42})=q^5(q-1)A_{00}(A_{20})=q^5(q-1)A_0(A_2)=q^5(q-1).
		\end{array}
	\end{equation*}
\end{remark}

\begin{lemma}[$n=3$]\label{lemma2.19} Assume that $\lambda, \kappa \geq 3$. Then, we have the following equalities.
	\begin{enumerate}
		\item $\alpha(A_{\lambda\kappa0},A_{\lambda\kappa0})=q^{8}(1+\dfrac{1}{q})\alpha(A_{\lambda-2,\kappa-2},A_{\lambda-2,\kappa-2})$.\\
		
		\item $\alpha(A_{\lambda\kappa1},A_{\lambda\kappa1})=q^{13}(1+\dfrac{1}{q})\alpha(A_{\lambda-2,\kappa-2},A_{\lambda-2,\kappa-2})$.\\
		
			\item $\alpha(A_{\lambda\kappa2},A_{\lambda\kappa2})=q^{18}(1+\dfrac{1}{q})\alpha(A_{\lambda-2,\kappa-2},A_{\lambda-2,\kappa-2})$.\\
			
			\item $\alpha(A_{\lambda20},A_{\lambda20})=q^{8}(1+\dfrac{1}{q})^2\alpha(A_{\lambda-2},A_{\lambda-2})$.\\
			
			\item $\alpha(A_{\lambda21},A_{\lambda21})=q^{13}(1+\dfrac{1}{q})^2\alpha(A_{\lambda-2},A_{\lambda-2})$.\\
			
			\item $\alpha(A_{\lambda22},A_{\lambda22})=q^{18}(1+\dfrac{1}{q})(1-\dfrac{1}{q^2})\alpha(A_{\lambda-2},A_{\lambda-2})$.\\
			
			\item $\alpha(A_{220},A_{220})=q^{8}(1+\dfrac{1}{q})^2(1-\dfrac{1}{q^2})$.\\
			
			\item $\alpha(A_{221},A_{221})=q^{13}(1+\dfrac{1}{q})^2(1-\dfrac{1}{q^2})$.\\
			
			\item $\alpha(A_{222},A_{222})=q^{18}(1+\dfrac{1}{q})(1-\dfrac{1}{q^2})(1+\dfrac{1}{q^3})$.\\
			
			\item $\alpha(A_{\lambda00},A_{\lambda00})=q^{2}(1+\dfrac{1}{q})(1-\dfrac{1}{q^2})\alpha(A_{\lambda-2},A_{\lambda-2})$.\\
			
			\item $\alpha(A_{\lambda10},A_{\lambda10})=q^{5}(1+\dfrac{1}{q})^2\alpha(A_{\lambda-2},A_{\lambda-2})$.\\
			
			\item $\alpha(A_{\lambda11},A_{\lambda11})=q^{10}(1+\dfrac{1}{q})(1-\dfrac{1}{q^2})\alpha(A_{\lambda-2},A_{\lambda-2})$.\\
			
			\item $\alpha(A_{200},A_{200})=q^{2}(1+\dfrac{1}{q})^2(1-\dfrac{1}{q^2})$.\\
			
			\item $\alpha(A_{210},A_{210})=q^{5}(1+\dfrac{1}{q})^3$.\\
			
			\item $\alpha(A_{211},A_{211})=q^{10}(1+\dfrac{1}{q})^2(1-\dfrac{1}{q^2})$.\\
			
			\item $\alpha(A_{000},A_{000})=(1+\dfrac{1}{q})(1-\dfrac{1}{q^2})(1+\dfrac{1}{q^3})$.\\
			
			\item $\alpha(A_{100},A_{100})=q^{1}(1+\dfrac{1}{q})^2(1-\dfrac{1}{q^2})$.\\
			
			\item $\alpha(A_{110},A_{110})=q^{4}(1+\dfrac{1}{q})^2(1-\dfrac{1}{q^2})$.\\
			
			\item $\alpha(A_{111},A_{111})=q^{9}(1+\dfrac{1}{q})(1-\dfrac{1}{q^2})(1+\dfrac{1}{q^3})$.\\
	\end{enumerate}
\end{lemma}

\begin{proof}
	This follows directly from \cite[Lemma 4.6]{Cho3}.
\end{proof}

\begin{corollary}[$n=3$]\label{corollary2.19}
	Assume that $\lambda, \kappa \geq 3$. For any $B \in X_3(O_E)$ such that $\pi^{-1}B$ is integral, the following equalities hold.
	\begin{enumerate}
		\item $A_{000}(B)=qA_{110}(B)+A_{211}(B)-q^3A_{222}(B)$\\
	\text{ }\text{ }\text{ }\text{ }\text{ }\text{ }\text{ }\text{ }\text{ }\text{ }\text{ }$=(q+1)A_{211}(B)-q^3A_{222}(B)$.\\
				
		\item $A_{100}(B)=(q^3+1)A_{111}(B)-qA_{221}(B)$.\\
		\item $A_{110}(B)=A_{211}(B)$.\\
		\item $A_{\lambda\kappa0}(B)=q^4A_{\lambda\kappa2}(B)$.\\
		\item $A_{\lambda20}(B)=q^3(q-1)A_{\lambda22}(B)$.\\
		\item $A_{220}(B)=q^2(q^2-q+1)A_{222}(B)$.\\
		\item $A_{\lambda00}(B)=q^2(q+1)A_{\lambda11}(B)-q^5A_{\lambda22}(B)$.\\
		\item $A_{200}(B)=q^2(q+1)A_{211}(B)-q^3(q^2-q+1)A_{222}(B)$.\\
		\item $A_{210}(B)=q(q-1)A_{221}(B)$.\\
		\item $A_{\lambda10}(B)=q^2A_{\lambda21}(B)$.
		
	\end{enumerate}
\end{corollary}
\begin{proof}
	This follows directly from Lemma \ref{lemma2.19} and Theorem \ref{theorem2.5}.
	\end{proof}

\begin{example}
	This Corollary \ref{corollary2.19} gives an efficient way to compute representation densities for $n=3$. For example, let us compute $A_{200}(A_{332})$. By Corollary \ref{corollary2.19} we have the following formula.
	
	\begin{equation*}\begin{array}{ll}
	A_{200}(A_{332})&=q^2(q+1)A_{211}(A_{332})-q^3(q^2-q+1)A_{222}(A_{332})\\
	&=q^2(q+1)A_{100}(A_{221})-q^3(q^2-q+1)A_{000}(A_{110})
	\end{array}
\end{equation*}

Also,
\begin{equation*}\begin{array}{l}
		A_{000}(A_{110})=A_{00}(A_{11})=(q+1)A_{11}(A_{11})=(q+1),
	\end{array}
\end{equation*}
and
\begin{equation*}\begin{array}{ll}
A_{100}(A_{221})&=(q^3+1)A_{111}(A_{221})-qA_{221}(A_{221})\\
&=(q^3+1)A_{000}(A_{110})-q\\
&=(q^3+1)A_{00}(A_{11})-q\\
&=(q^3+1)(q+1)-q\\
&=q^4+q^3+1.
\end{array}
\end{equation*}

Therefore,
\begin{equation*}\begin{array}{ll}
		A_{200}(A_{332})&=q^2(q+1)(q^4+q^3+1)-q^3(q^2-q+1)(q+1)\\
		&=(q+1)(q^6+q^4-q^3+q^2).
	\end{array}
\end{equation*}
\end{example}

\begin{lemma}[$n=4$]\label{lemma2.20} Assume that $\lambda, \kappa, \mu \geq 3$. Then, we have the following equalities.
	\begin{enumerate}
	\item $\alpha(A_{\lambda\kappa\mu0},A_{\lambda\kappa\mu0})=q^{18}(1+\dfrac{1}{q})\alpha(A_{\lambda-2,\kappa-2,\mu-2},A_{\lambda-2,\kappa-2,\mu-2})$.\\
		\item $\alpha(A_{\lambda\kappa\mu1},A_{\lambda\kappa\mu1})=q^{25}(1+\dfrac{1}{q})\alpha(A_{\lambda-2,\kappa-2,\mu-2},A_{\lambda-2,\kappa-2,\mu-2})$.\\
			\item $\alpha(A_{\lambda\kappa\mu2},A_{\lambda\kappa\mu2})=q^{32}(1+\dfrac{1}{q})\alpha(A_{\lambda-2,\kappa-2,\mu-2},A_{\lambda-2,\kappa-2,\mu-2})$.\\
			
	\item $\alpha(A_{\lambda\kappa20},A_{\lambda\kappa20})=q^{18}(1+\dfrac{1}{q})^2\alpha(A_{\lambda-2,\kappa-2},A_{\lambda-2,\kappa-2})$.\\
	
		\item $\alpha(A_{\lambda\kappa21},A_{\lambda\kappa21})=q^{25}(1+\dfrac{1}{q})^2\alpha(A_{\lambda-2,\kappa-2},A_{\lambda-2,\kappa-2})$.\\
		
			\item $\alpha(A_{\lambda\kappa22},A_{\lambda\kappa22})=q^{32}(1+\dfrac{1}{q})(1-\dfrac{1}{q^2})\alpha(A_{\lambda-2,\kappa-2},A_{\lambda-2,\kappa-2})$.\\
			
			\item $\alpha(A_{\lambda220},A_{\lambda220})=q^{18}(1+\dfrac{1}{q})^2(1-\dfrac{1}{q^2})\alpha(A_{\lambda-2},A_{\lambda-2})$.\\
			
			\item $\alpha(A_{\lambda221},A_{\lambda221})=q^{25}(1+\dfrac{1}{q})^2(1-\dfrac{1}{q^2})\alpha(A_{\lambda-2},A_{\lambda-2})$.\\
			
			\item $\alpha(A_{\lambda222},A_{\lambda222})=q^{32}(1+\dfrac{1}{q})(1-\dfrac{1}{q^2})(1+\dfrac{1}{q^3})\alpha(A_{\lambda-2},A_{\lambda-2})$.\\
			\item
			$\alpha(A_{2220},A_{2220})=q^{18}(1+\dfrac{1}{q})^2(1-\dfrac{1}{q^2})(1+\dfrac{1}{q^3})$.\\
			
			\item
			$\alpha(A_{2221},A_{2221})=q^{25}(1+\dfrac{1}{q})^2(1-\dfrac{1}{q^2})(1+\dfrac{1}{q^3})$.\\
			
			\item
			$\alpha(A_{2222},A_{2222})=q^{32}(1+\dfrac{1}{q})(1-\dfrac{1}{q^2})(1+\dfrac{1}{q^3})(1-\dfrac{1}{q^4})$.\\
			
			\item $\alpha(A_{\lambda\kappa00},A_{\lambda\kappa00})=q^8(1+\dfrac{1}{q})(1-\dfrac{1}{q^2})\alpha(A_{\lambda-2,\kappa-2},A_{\lambda-2,\kappa-2})$.\\
			
			\item $\alpha(A_{\lambda\kappa10},A_{\lambda\kappa10})=q^{13}(1+\dfrac{1}{q})^2\alpha(A_{\lambda-2,\kappa-2},A_{\lambda-2,\kappa-2})$.\\
			
			\item $\alpha(A_{\lambda\kappa11},A_{\lambda\kappa11})=q^{20}(1+\dfrac{1}{q})(1-\dfrac{1}{q^2})\alpha(A_{\lambda-2,\kappa-2},A_{\lambda-2,\kappa-2})$.\\

			\item $\alpha(A_{\lambda200},A_{\lambda200})=q^8(1+\dfrac{1}{q})^2(1-\dfrac{1}{q^2})\alpha(A_{\lambda-2},A_{\lambda-2})$.\\
			
			\item $\alpha(A_{\lambda210},A_{\lambda210})=q^{13}(1+\dfrac{1}{q})^3\alpha(A_{\lambda-2},A_{\lambda-2})$.\\
			
			\item $\alpha(A_{\lambda211},A_{\lambda211})=q^{20}(1+\dfrac{1}{q})^2(1-\dfrac{1}{q^2})\alpha(A_{\lambda-2},A_{\lambda-2})$.\\
			
			\item $\alpha(A_{2200},A_{2200})=q^8(1+\dfrac{1}{q})^2(1-\dfrac{1}{q^2})^2$.\\
			
			\item $\alpha(A_{2210},A_{2210})=q^{13}(1+\dfrac{1}{q})^3(1-\dfrac{1}{q^2})$.\\
			
			\item $\alpha(A_{2211},A_{2211})=q^{20}(1+\dfrac{1}{q})^2(1-\dfrac{1}{q^2})^2$.\\

			\item $\alpha(A_{\lambda000},A_{\lambda000})=q^{2}(1+\dfrac{1}{q})(1-\dfrac{1}{q^2})(1+\dfrac{1}{q^3})\alpha(A_{\lambda-2},A_{\lambda-2})$.\\
			
			\item $\alpha(A_{\lambda100},A_{\lambda100})=q^{5}(1+\dfrac{1}{q})^2(1-\dfrac{1}{q^2})\alpha(A_{\lambda-2},A_{\lambda-2})$.\\
			
			\item $\alpha(A_{\lambda110},A_{\lambda110})=q^{10}(1+\dfrac{1}{q})^2(1-\dfrac{1}{q^2})\alpha(A_{\lambda-2},A_{\lambda-2})$.\\
			
			\item $\alpha(A_{\lambda111},A_{\lambda111})=q^{17}(1+\dfrac{1}{q})(1-\dfrac{1}{q^2})(1+\dfrac{1}{q^3})\alpha(A_{\lambda-2},A_{\lambda-2})$.\\

			\item $\alpha(A_{2000},A_{2000})=q^{2}(1+\dfrac{1}{q})^2(1-\dfrac{1}{q^2})(1+\dfrac{1}{q^3})$.\\
			
			\item $\alpha(A_{2100},A_{2100})=q^{5}(1+\dfrac{1}{q})^3(1-\dfrac{1}{q^2})$.\\
			
			\item $\alpha(A_{2110},A_{2110})=q^{10}(1+\dfrac{1}{q})^3(1-\dfrac{1}{q^2})$.\\
			
			\item $\alpha(A_{2111},A_{2111})=q^{17}(1+\dfrac{1}{q})^2(1-\dfrac{1}{q^2})(1+\dfrac{1}{q^3})$.\\

			\item $\alpha(A_{0000},A_{0000})=(1+\dfrac{1}{q})(1-\dfrac{1}{q^2})(1+\dfrac{1}{q^3})(1-\dfrac{1}{q^4})$.\\
			
			\item $\alpha(A_{1000},A_{1000})=q(1+\dfrac{1}{q})^2(1-\dfrac{1}{q^2})(1+\dfrac{1}{q^3})$.\\
			
			\item $\alpha(A_{1100},A_{1100})=q^4(1+\dfrac{1}{q})^2(1-\dfrac{1}{q^2})^2$.\\
			
			\item $\alpha(A_{1110},A_{1110})=q^9(1+\dfrac{1}{q})^2(1-\dfrac{1}{q^2})(1+\dfrac{1}{q^3})$.\\
			
			\item $\alpha(A_{1111},A_{1111})=q^{16}(1+\dfrac{1}{q})(1-\dfrac{1}{q^2})(1+\dfrac{1}{q^3})(1-\dfrac{1}{q^4})$.\\
	\end{enumerate}
\end{lemma}
\begin{proof}
	This follows directly from \cite[Lemma 4.6]{Cho3}.
\end{proof}
\begin{corollary}[$n=4$]\label{corollary2.20}
		Assume that $\lambda, \kappa, \mu \geq 3$. For any $B \in X_4(O_E)$ such that $\pi^{-1}B$ is integral, the following equalities hold.
		
		\begin{enumerate}
			\item $A_{\lambda\kappa\mu0}(B)=q^6A_{\lambda\kappa\mu2}(B)$.\\
			
			\item $A_{\lambda\kappa20}(B)=q^5(q-1)A_{\lambda\kappa22}(B)$.\\
			
			\item $A_{\lambda\kappa10}(B)=q^4A_{\lambda\kappa21}(B)$.\\
			
			\item $A_{\lambda220}(B)=q^4(q^2-q+1)A_{\lambda222}(B)$.\\
			
			\item $A_{\lambda210}(B)=q^3(q-1)A_{\lambda221}(B)$.\\
			
			\item $A_{\lambda110}(B)=q^2A_{\lambda211}(B)$.\\
			
			\item $A_{2220}(B)=q^3(q^3-q^2+q-1)A_{2222}(B)$.\\
			
			\item $A_{2210}(B)=q^2(q^2-q+1)A_{2221}(B)$.\\
			
			\item $A_{2110}(B)=q(q-1)A_{2211}(B)$.\\
			
			\item $A_{1110}(B)=A_{2111}(B)$.\\
			
			\item $A_{\lambda\kappa00}(B)=q^4(q+1)A_{\lambda\kappa11}(B)-q^9A_{\lambda\kappa22}(B)$.\\
			
			\item $A_{\lambda200}(B)=q^4(q+1)A_{\lambda211}(B)-q^7(q^2-q+1)A_{\lambda222}(B)$.\\
			
			\item $A_{\lambda100}(B)=q^2(q^3+1)A_{\lambda111}(B)-q^5A_{\lambda221}(B)$.\\
			\item $A_{2200}(B)=q^4(q+1)A_{2211}(B)-q^5(q^2-q+1)(q^2+1)A_{2222}(B)$.\\
			
			\item $A_{2100}(B)=q^2(q^3+1)A_{2111}(B)-q^3(q^2-q+1)A_{2221}(B)$.\\
			
			\item $A_{1100}(B)=(q^3+1)(q^2+1)A_{1111}(B)-qA_{2211}(B)$.\\
			
			\item $A_{\lambda000}(B)=q^3A_{\lambda110}(B)+q^4A_{\lambda211}(B)-q^9A_{\lambda222}(B)$\\
			\text{ }\text{ }\text{ }\text{ }\text{ }\text{ }\text{ }\text{ }\text{ }\text{ }\text{ }\text{ }$=q^4(q+1)A_{\lambda211}(B)-q^9A_{\lambda222}(B)$.\\
			
			\item $A_{2000}(B)=q^3A_{2110}(B)+q^3(q-1)A_{2211}(B)$\\
			\text{ }\text{ }\text{ }\text{ }\text{ }\text{ }\text{ }\text{ }\text{ }\text{ }\text{ }\text{ }\text{ }\text{ }\text{ }\text{ }$-q^6(q^3-q^2+q-1)A_{2222}(B)$
			\quad\\
			\text{ }\text{ }\text{ }\text{ }\text{ }\text{ }\text{ }
			$=q^3(q^2-1)A_{2211}(B)-q^6(q^3-q^2+q-1)A_{2222}(B)$.\\

			\item $A_{1000}(B)=q(q^2-q+1)A_{1110}(B)+(q^2-q+1)A_{2111}(B)-q^3A_{2221}(B)$
			\quad\\
			\text{ }\text{ }\text{ }\text{ }\text{ }\text{ }\text{ }\text{ }\text{ }\text{ }\text{ }
			$=(q^3+1)A_{2111}(B)-q^3A_{2221}(B)$.\\
			
			\item $A_{0000}(B)=qA_{1100}(B)-(q^6-1)A_{1111}(B)-qA_{2211}(B)+q^6A_{2222}(B)$
			\quad\\
			\text{ }\text{ }\text{ }\text{ }\text{ }\text{ }\text{ }\text{ }\text{ }\text{ }\text{ }
			$=(q+1)(q^3+1)A_{1111}(B)-q(q+1)A_{2211}(B)+q^6A_{2222}(B)$.
			
		\end{enumerate}
\end{corollary}
\begin{proof}
This follows directly from Lemma \ref{lemma2.20} and Theorem \ref{theorem2.5}.
\end{proof}

\begin{remark}
	The most important advantage of this Corollary \ref{corollary2.20} is that we can reduce the computation of representation densities for $B$ to the ones for $\pi^{-1}B$. For example, to compute $A_{1000}(B)$, we only need to compute $(q^3+1)A_{1000}(\pi^{-1}B)-q^3A_{1110}(\pi^{-1}B)$. In Section \ref{section3}, we will provide several examples of computations of representation densities by using this Corollary.
\end{remark}

\section{Special cycles and representation densities}\label{section3}
In this section, we apply Theorem \ref{theorem2.5} to understand the arithmetic intersection theory of special cycles on unitary Shimura varieties.

\subsection{Special cycles on unitary Rapoport-Zink spaces}
In this subsection, we recall some facts and definitions of unitary Rapoport-Zink spaces $\CN^h_{E/F}(1,n-1)$ and special cycles $\CZ(x)$ from \cite{Cho} and \cite{Cho2}.
In this section, we assume that $F$ is an unramified extension of $\BQ_p$. We fix integers $n$ and $0 \leq h \leq n$. Let $\breve{E}$ be the completion of a maximal unramified extension of $E$ and let $O_{\breve{E}}$ be its ring of integers. We denote by $k$ the residue field of $O_{\breve{E}}$.

Let $(\BX,i_{\BX},\lambda_{\BX})$ be the following triple.
\begin{enumerate}
	\item $\BX$ is a supersingular strict formal $O_F$-module of $F$-height $2n$ over $\BF_{q^2}$.
	
	\item $i_{\BX}:O_E \rightarrow \End\BX$ is an $O_E$-action on $\BX$ that extends the $O_F$-action on $\BX$. This satisfies the following signature condition of $(1,n-1)$: for all $a \in O_E$, we have
	\begin{equation*}
		\Charpol(i_{\BX}(a),\Lie \BX)=(T-a)(T-a^*)^{n-1}.
	\end{equation*}
	
	\item $\lambda_{\BX}:\BX \rightarrow \BX^{\vee}$ is a polarization such that the corresponding Rosati involution induces $^*$ on $O_E$. We assume that $\ker \lambda_{\BX} \subset \BX[\pi]$ and its order is $q^{2h}$.
\end{enumerate}

This triple is unique up to quasi-isogeny.

Let (Nilp) be the category of $O_{\breve{E}}$-schemes $S$ such that $\pi$ is locally nilpotent on $S$. We define $\CN^h_{E/F}(1,n-1)$ as a functor sending $S$ to the set of isomorphism classes of tuples $(X,i_X,\lambda_X,\rho_X)$.

Here, $X$ is a supersingular strict formal $O_F$-module of $F$-height $2n$ over $S$ and $i_X$ is an $O_E$-action on $X$ satisfying the signature condition of $(1,n-1)$. $\lambda_X:X \rightarrow X^{\vee}$ is a polarization such that its Rosati involution induces $^*$ on $E$. $\rho_X:X_{\overline{S}} \rightarrow \BX \times_{\BF_{q^2}}\overline{S}$ is an $O_E$-linear quasi-isogeny of height $0$, where $\overline{S}=S \times_{O_E} \BF_{q^2}$. We also assume that the following diagram commutes up to a constant in $O_F^{\times}$:
\begin{center}
	\begin{tikzcd}
		X_{\overline{S}} \arrow{r}{\lambda_{X_{\overline{S}}}}
		\arrow{d}{\rho_X}
		&X^{\vee}_{\overline{S}} \\
		\mathbb{X}_{\overline{S}}
		\arrow{r}{\lambda_{\BX_{\overline{S}}}}  &\mathbb{X}_{\overline{S}}^{\vee}
		\arrow{u}{\rho_X^{\vee}}.
	\end{tikzcd}
\end{center}

Then $\CN^h_{E/F}(1,n-1)$ is representable by a formal scheme over $\Spf O_{\breve{E}}$ which is locally formally of finite type and regular (\cite[Proposition 3.33]{Cho}).

To define special cycles on $\CN:=\CN^h_{E/F}(1,n-1)$, we need to fix a triple $(\overline{\BY},i_{\overline{\BY}},\lambda_{\overline{\BY}})$: Let $\overline{\BY}$ be a strict formal $O_{F}$-module of $F$-height 2 over $k$, and let $i_{\overline{\BY}}$ be an $O_{E}$-action satisfying the signature condition of signature $(0,1)$. Also, $\lambda_{\overline{\BY}}$ is a principal polarization. Then, we can consider the corresponding moduli space $\CN^0(0,1)$.

We define $\BV=\Hom_{O_E}(\overline{\BY},\BX) \otimes_{\BZ} \BQ$ to be the space of special homomorphisms. For each $x,y \in \BV$, we define the hermitian form $h$ on $\BV$ as
\begin{equation*}
		h(x,y)=\lambda_{\overline{\BY}}^{-1} \circ y^{\vee}  \circ \lambda_{\BX} \circ x \in \End_{O_{E}}(\overline{\BY}) \otimes \BQ \overset{i_{\overline{\BY}}^{-1}}{\simeq} E.
	\end{equation*}

For $x \in \BV$, we define the special cycles $\CZ(x)$ to be the closed formal scheme of $\CN^0 \times \CN$ as follows: For each $S$ in (Nilp), $\CZ(x)(S)$ is the set of all points $(\overline{Y},i_{\overline{Y}},\lambda_{\overline{Y}},\rho_{\overline{Y}},X,i_{X},\lambda_X,\rho_X)$ in $\CN^0\times \CN(S)$ such that the quasi-homomorphism
\begin{equation*}
	\rho_X^{-1}\circ x \circ \rho_{\overline{Y}}:\overline{Y} \times_{S} \overline{S} \rightarrow X \times_{S} \overline{S}
\end{equation*}
extends to a homomorphism from $\overline{Y}$ to $X$.

For a basis $\lbrace x_1, x_2, \dots, x_n \rbrace$ of $\BV$, we denote by $\langle \CZ(x_1), \CZ(x_2), \dots, \CZ(x_n) \rangle$ the arithmetic intersection number in $\CN=\CN^h_{E/F}(1,n-1)$:
\begin{equation*}
\langle \CZ(x_1), \CZ(x_2), \dots, \CZ(x_n) \rangle:=\chi(O_{\CZ(x_1)} \otimes_{\CN}^{\BL} \dots \otimes_{\CN}^{\BN}O_{\CZ(x_n)}).
\end{equation*}
Here $\chi$ is the Euler-Poincare characteristic and $\otimes_{\BL}$ is the derived tensor product of $O_{\CN}$-modules.

Now, we can state the following conjecture from \cite[Conjecture 3.17, Conjecture 3.25]{Cho2} and \cite[Conjecture 4.9]{Cho3}.

\begin{conjecture}\label{conjecture3.1}(\cite[Conjecture 3.17, Conjecture 3.25]{Cho2}, \cite[Conjecture 4.9]{Cho3}) For a basis $\lbrace x_1, \dots, x_{2n}\rbrace$ of $\BV$, and special cycles $\CZ(x_1),$ $\dots,$ $\CZ(x_{2n})$ in $\CN^n_{E/F}(1,2n-1)$, we have
	\begin{equation*}
		\langle \CZ(x_1),\dots,\CZ(x_{2n}) \rangle=\sum_{\lambda \in \Lambda_{2n}^+} D_{\lambda} A_{\lambda}(B)-\sum_{0 \leq i \leq n-1}\Fb_i^0 A_{(1^i,0^{2n-i})}(B).
	\end{equation*}
Here $B$ is the $2n \times 2n$ hermitian matrix $(h(x_i,x_j))$ in $X_{2n}(O_E)$ and $D_{\lambda}, \Fb_i^0$ are some constants (we refer to \cite{Cho3} for the precise definitions of these constants).
	\end{conjecture}

\begin{remark}
	Conjecture \ref{conjecture3.1} gives conjectural formulas for the arithmetic intersection numbers of special cycles in $\CN^h_{E/F}(1,n-1)$ for any $h$ after some reductions. Indeed, \cite[Conjecture 3.17]{Cho2} and \cite[Conjecture 3.25]{Cho2} are not only about special cycles $\CZ$, these are also about special cycles $\CY$ (even though we will not discuss $\CY$-cycles here). In the case that $h=0$, this is the Kudla-Rapoport conjecture \cite[Conjecture 1.3]{KR2} and is proved in \cite{LZ}. The case when $h=1$ is also proved in \cite{LZ}.
\end{remark}

\subsection{Some applications}\label{subsection3.2}
In this subsection, we apply Theorem \ref{theorem2.5} to compute some arithmetic intersection numbers of special cycles and suggest a way to understand these.

\begin{example}\label{example3.3}[$\CN^1_{E/F}(1,1)$]
	In this example, we consider the arithmetic intersection numbers of special cycles in $\CN^1_{E/F}(1,1)$ and answer to the question raised in \cite[Remark 4.11]{Cho3}. In fact, this is our main motivation for Theorem \ref{theorem2.5}.
	In \cite{San}, Sankaran proved that for special homomorphisms $x,y \in \BV$, the arithmetic intersection numbers of special cycles $\langle \CZ(x),\CZ(y) \rangle$ in $\CN^1_{E/F}(1,1)$ is
	\begin{equation*}
	\dfrac{q}{(q+1)^2}\lbrace \alpha'(A_{10},B)-\dfrac{q^2}{q^2-1} \alpha(A_{00},B)\rbrace,
	\end{equation*}
where $B=\left( \begin{array}{cc} h(x,x) & h(x,y)\\ h(y,x) & h(y,y) \end{array}\right)$.
Let us write $A_{\alpha\beta}(x,y)$ for $A_{\alpha\beta}(B)$.
In terms of weighted lattice counting (i.e a version of the Cho-Yamauchi's formula \cite{CY}) in \cite{Cho3}, this can be written as
\begin{equation*}\begin{array}{ll}
		\langle \CZ(x),\CZ(y) \rangle&=-(q-1) A_{11}(x,y)-(q^2-1)\mathlarger{\sum}_{\lambda, \kappa \geq 2} A_{\lambda\kappa}(x,y)\\
		&+\mathlarger{\sum}_{\lambda \geq 2} A_{\lambda1}(x,y)+\mathlarger{\sum}_{\lambda \geq 2} A_{\lambda0}(x,y).
	\end{array}
\end{equation*}

Therefore, we have that
\begin{equation}\label{eq3.2.1}
	\begin{array}{l}
	\langle	\CZ(x),\CZ(y) \rangle -\langle \CZ(\dfrac{x}{\pi}),\CZ(\dfrac{y}{\pi}) \rangle\\
	=-(q-1) A_{11}(x,y)-(q^2-1)\mathlarger{\sum}_{\lambda, \kappa \geq 2} A_{\lambda\kappa}(x,y)
	+\mathlarger{\sum}_{\lambda \geq 2} A_{\lambda1}(x,y)+\mathlarger{\sum}_{\lambda \geq 2} A_{\lambda0}(x,y)\\\\
	+(q-1) A_{11}(\dfrac{x}{\pi},\dfrac{y}{\pi})+(q^2-1)\mathlarger{\sum}_{\lambda, \kappa \geq 2} A_{\lambda\kappa}(\dfrac{x}{\pi},\dfrac{y}{\pi}))
	-\mathlarger{\sum}_{\lambda \geq 2} A_{\lambda1}(\dfrac{x}{\pi},\dfrac{y}{\pi})-\mathlarger{\sum}_{\lambda \geq 2} A_{\lambda0}(\dfrac{x}{\pi},\dfrac{y}{\pi})
\end{array}\end{equation}\begin{equation*}\begin{array}{l}
	=-(q-1) A_{11}(x,y)-(q^2-1)\mathlarger{\sum}_{\lambda, \kappa \geq 2} A_{\lambda\kappa}(x,y)
	+\mathlarger{\sum}_{\lambda \geq 2} A_{\lambda1}(x,y)+\mathlarger{\sum}_{\lambda \geq 2} A_{\lambda0}(x,y)\\\\
	+(q-1) A_{33}(x,y)+(q^2-1)\mathlarger{\sum}_{\lambda, \kappa \geq 4} A_{\lambda\kappa}(x,y)
	-\mathlarger{\sum}_{\lambda \geq 4} A_{\lambda3}(x,y)-\mathlarger{\sum}_{\lambda \geq 4} A_{\lambda2}(x,y)
\end{array}\end{equation*}\begin{equation*}\begin{array}{l}
	=-(q-1)A_{11}(x,y)-(q^2-1)A_{22}(x,y)-(q^2-q)A_{33}(x,y)+A_{20}(x,y)+A_{31}(x,y)\\\\
	+\mathlarger{\sum}_{\lambda \geq 4} A_{\lambda0}(x,y)-q^2\mathlarger{\sum}_{\lambda \geq 4} A_{\lambda2}(x,y)+\mathlarger{\sum}_{\lambda \geq 4} A_{\lambda1}(x,y)-q^2\mathlarger{\sum}_{\lambda \geq 4} A_{\lambda3}(x,y)
\end{array}
\end{equation*}

On the other hand, by \cite[Theorem 3.14]{San2}, we know that for $y\in \BV$ with $\val(h(y,y)) \geq 2$, we have
\begin{equation*}
	\CZ(y)-\CZ(\dfrac{y}{\pi})=\sum_{\substack{y/\pi \in \Lambda\\ \Lambda \simeq A_{-1,-1}}} \BP_{\Lambda}^1+\sum_{\substack{y/\pi \in \Lambda\\ \Lambda \simeq A_{00}}} \BP_{\Lambda}^1
\end{equation*}

Here $\Lambda \simeq A_{\alpha\beta}$ means that the lattice $\Lambda$ has the hermitian matrix $A_{\alpha\beta}$.

Also, by \cite[Lemma 2.11]{San}, we have that
\begin{equation*}
	\langle \CZ(x), \BP_{\Lambda}^1 \rangle=-q 1_{\Lambda}(\dfrac{x}{\pi}),
\end{equation*}
if $\Lambda \simeq A_{-1,-1}$, and
\begin{equation*}
	\langle \CZ(x), \BP_{\Lambda}^1 \rangle=1_{\Lambda}(x),
\end{equation*}
if $\Lambda \simeq A_{00}$.

Therefore, we have that
\begin{equation*}\begin{array}{ll}
	\langle \CZ(x),\CZ(y)-\CZ(\dfrac{y}{\pi}) \rangle&=-qA_{-1,-1}(\dfrac{x}{\pi},\dfrac{y}{\pi})+A_{00}(x,\dfrac{y}{\pi})\\
	&=-qA_{11}(x,y)+A_{00}(x,\dfrac{y}{\pi}).
	\end{array}
\end{equation*}

Also, we have that
\begin{equation*}\begin{array}{l}
\langle	\CZ(x),\CZ(y) \rangle -\langle \CZ(\dfrac{x}{\pi}),\CZ(\dfrac{y}{\pi}) \rangle\\\\
=\langle \CZ(x),\CZ(y)-\CZ(\dfrac{y}{\pi})\rangle+\langle \CZ(x)-\CZ(\dfrac{x}{\pi}),\CZ(\dfrac{y}{\pi})\rangle
\\\\

=-qA_{11}(x,y)+A_{00}(x,\dfrac{y}{\pi})-qA_{11}(x,\dfrac{y}{\pi})+A_{00}(\dfrac{x}{\pi},\dfrac{y}{\pi})\\\\

=-qA_{11}(x,y)+A_{00}(x,\dfrac{y}{\pi})-qA_{11}(x,\dfrac{y}{\pi})+A_{22}(x,y).
\end{array}
\end{equation*}

Note that this equation looks very different from \eqref{eq3.2.1}.

Combining this and \eqref{eq3.2.1}, we have that
\begin{equation}\label{eq3.2.2}
	\begin{array}{l}
		-qA_{11}(x,y)+A_{00}(x,\dfrac{y}{\pi})-qA_{11}(x,\dfrac{y}{\pi})+A_{22}(x,y)\\
		=-(q-1)A_{11}(x,y)-(q^2-1)A_{22}(x,y)-(q^2-q)A_{33}(x,y)+A_{20}(x,y)+A_{31}(x,y)\\\\
		+\mathlarger{\sum}_{\lambda \geq 4} A_{\lambda0}(x,y)-q^2\mathlarger{\sum}_{\lambda \geq 4} A_{\lambda2}(x,y)+\mathlarger{\sum}_{\lambda \geq 4} A_{\lambda1}(x,y)-q^2\mathlarger{\sum}_{\lambda \geq 4} A_{\lambda3}(x,y)
	\end{array}
\end{equation}

Why do we have this equality? This question is the starting point for our Theorem \ref{theorem2.5}.

Note that $A_{00}(x,y)-qA_{11}(x,y)=1$ for any $x,y$ such that $\left(\begin{array}{cc}h(x,x)&h(x,y)\\h(y,x)&h(y,y)\end{array}\right)$ is integral. Therefore, we have that
\begin{equation*}
	A_{00}(x,y)-qA_{11}(x,y)=A_{00}(x,\dfrac{y}{\pi})-qA_{11}(x,\dfrac{y}{\pi}).
\end{equation*}

This implies that \eqref{eq3.2.2} is equivalent to
\begin{equation*}\begin{array}{ll}
	0&=A_{00}(x,y)-(q+1)A_{11}(x,y)+q^2A_{22}(x,y)-A_{20}(x,y)\\
	&+(q^2-q)A_{33}(x,y)-A_{31}(x,y)\\
	&-\mathlarger{\sum}_{\lambda \geq 4} \lbrace A_{\lambda0}(x,y)-q^2A_{\lambda2}(x,y)\rbrace-\mathlarger{\sum}_{\lambda \geq 4} \lbrace A_{\lambda1}(x,y)-q^2 A_{\lambda3}(x,y)\rbrace\\\\
	
	=&\lbrace A_{00}(x,y)-(q+1)A_{11}(x,y)+qA_{22}(x,y)\rbrace\\\\
	&-\lbrace A_{20}-q(q-1)A_{22}(x,y)\rbrace-\lbrace A_{31}(x,y)-q(q-1)A_{33}(x,y)\rbrace\\\\
	&-\mathlarger{\sum}_{\lambda \geq 4} \lbrace A_{\lambda0}(x,y)-q^2A_{\lambda2}(x,y)\rbrace-\mathlarger{\sum}_{\lambda \geq 4} \lbrace A_{\lambda1}(x,y)-q^2 A_{\lambda3}(x,y)\rbrace.
	
	\end{array}
\end{equation*}

Now, Corollary \ref{corollary2.17} gives the reason why this equality holds for any $x,y$ such that $\dfrac{1}{\pi^2}\left(\begin{array}{cc}h(x,x)&h(x,y)\\h(y,x)&h(y,y)\end{array}\right)$ is integral.
\end{example}

\begin{example}\label{example3.4}[$\CN^0_{E/F}(1,3)$] Assume that $\alpha > \beta \geq \gamma \geq \delta \geq 0$ and $\alpha$ is even. Let $L_{\beta\gamma\delta}$ be a lattice of rank $3$ in $\BV$ with the hermitian matrix $A_{\beta\gamma\delta}$ and let $x$ be a special homomorphism in $\BV$ with $\val(h(x,x))=\alpha$. Assume that $x$ is orthogonal to $L_{\beta\gamma\delta}$. By the Kudla-Rapoport conjecture \cite[Conjecture 1.3]{KR2} and its proof \cite{LZ}, we have that the arithmetic intersection number of special cycles $\langle \CZ(L_{\beta\gamma\delta}),\CZ(x)\rangle$ is $A_{0000}'(A_{\alpha\beta\gamma\delta})$. We will show that 
	\begin{equation}\label{eq3.2.3}\begin{array}{ll}
			A_{0000}'(A_{\alpha\beta\gamma\delta})&=A_{100}(A_{\beta\gamma\delta})A_{0000}'(A_{\alpha100})\\\\
			&+\mathlarger{\sum}_{\lambda \geq 3}A_{\lambda00}(A_{\beta\gamma\delta})\lbrace A_{0000}' (A_{\alpha\lambda00})-A_{0000}'(A_{\alpha(\lambda-2)00})\rbrace\\\\
			&+(1-q^2)\sum_{\lambda \geq \kappa \geq 1} A_{\lambda\kappa0}(A_{\beta\gamma\delta}) A_{1110}(A_{\alpha\lambda\kappa0})\\\\
			&+(1+q)(1-q^2)\sum_{\lambda \geq \kappa \geq \mu \geq 1} A_{\lambda\kappa\mu}(A_{\beta\gamma\delta}) A_{1110}(A_{\alpha\lambda\kappa\mu}).
		\end{array}
	\end{equation}
in some cases. This is modelled on the analytic part of \cite{LZ} (in particular, \cite[Theorem 4.2.1, Theorem 5.4.1]{LZ}). Note that this formula relates $A_{0000}'(A_{\alpha\beta\gamma\delta})$ to the counting of vertex lattices of type $3$ (i.e., its hermitian matrix is $A_{1110}$) plus additional parts (these correspond to the horizontal parts of $\CZ(L_{\beta\gamma\delta})$ by \cite[Theorem 4.2.1]{LZ}).

By the Cho-Yamauchi's formula in unitary case \cite[Corollary 3.5.3]{LZ}, we have
\begin{equation}\label{eq3.2.4}\begin{array}{ll}
		A_{0000}'(A_{\alpha\beta\gamma\delta})&=\mathlarger{\sum}_{\lambda \geq 1} A_{\lambda000}(A_{\alpha\beta\gamma\delta})+(1+q)\mathlarger{\sum}_{\lambda\geq \kappa \geq 1} A_{\lambda\kappa00}(A_{\alpha\beta\gamma\delta})\\
		&+(1+q)(1-q^2)\mathlarger{\sum}_{\lambda\geq \kappa \geq \mu\geq 1} A_{\lambda\kappa\mu0}(A_{\alpha\beta\gamma\delta})\\
		&+(1+q)(1-q^2)(1+q^3)\mathlarger{\sum}_{\lambda\geq \kappa \geq \mu \geq \eta \geq 1} A_{\lambda\kappa\mu\eta}(A_{\alpha\beta\gamma\delta}).
	\end{array}
\end{equation}

Now, let us show that \eqref{eq3.2.3} holds in the following cases.
\begin{enumerate}
	\item ($L_{\beta\gamma\delta}=L_{100}$ case) Indeed, in $L_{\beta\gamma\delta}=L_{\nu00}$ case, \eqref{eq3.2.3} holds since all $A_{\lambda00}(A_{\nu00})$ are $1$ for all $\nu \geq \lambda$. But, we will compute $A_{0000}'(A_{\alpha100})$ for later use. By \eqref{eq3.2.4} we have that
	\begin{equation*}\begin{array}{ll}
		A_{0000}'(A_{\alpha100})&=A_{1000}(A_{\alpha100})+\dots+A_{(\alpha-1)000}(A_{\alpha100})\\
		&+(1+q)\lbrace A_{2100}(A_{\alpha100})+\dots+A_{\alpha100}(A_{\alpha100)}\rbrace.
		\end{array}
	\end{equation*}

Also, by Corollary \ref{corollary2.20}, we have
\begin{flalign*}
	&\begin{cases}
A_{1000}(A_{\alpha100})=A_{10}(A_{\alpha1})=1,&\\
A_{3000}(A_{\alpha100})=A_{30}(A_{\alpha1})=q^2A_{32}(A_{\alpha1})=0,&\\
\quad\quad\quad\quad\vdots&\\
A_{(\alpha-1)000}(A_{\alpha100})=A_{(\alpha-1)0}(A_{\alpha1})=q^2A_{(\alpha-1)2}(A_{\alpha1})=0,
	\end{cases}&
\end{flalign*}
\begin{flalign*}
	&\begin{cases}
A_{2100}(A_{\alpha100})=A_{21}(A_{\alpha1})=A_{10}(A_{(\alpha-1)0})=1,&\\
\quad\quad\quad\quad\vdots&\\
A_{\alpha100}(A_{\alpha100})=1.
	\end{cases}&
\end{flalign*}
This implies that 
\begin{equation*}
	A_{0000}(A_{\alpha100})'=(q+1)\dfrac{\alpha}{2}+1.
\end{equation*}
Also \eqref{eq3.2.3} holds since $A_{100}(A_{100})=1$.\\

\item ($L_{\beta\gamma\delta}=L_{111}$ case) By \eqref{eq3.2.4} we have that
\begin{equation*}\begin{array}{ll}
		A_{0000}'(A_{\alpha111})&=A_{1000}(A_{\alpha111})+\dots+A_{(\alpha-1)000}(A_{\alpha111})\\
		&+(1+q)\lbrace A_{2100}(A_{\alpha111})+\dots+A_{\alpha111}(A_{\alpha111)}\rbrace\\
		&+(1+q)(1-q^2)\lbrace A_{1110}(A_{\alpha111})+\dots+A_{(\alpha-1)110}(A_{\alpha111})\rbrace\\
		&+(1+q)(1-q^2)(1+q^3)\lbrace A_{2111}(A_{\alpha111})+\dots+A_{\alpha111}(A_{\alpha111}).\rbrace
		\end{array}
\end{equation*}

Now, let us compute the following representation densities by using Corollary \ref{corollary2.20}.
\begin{flalign*}
	&\begin{cases}
	A_{2111}(A_{\alpha111})=A_{2}(A_{\alpha})=1,&\\
	\quad\quad\quad\quad\vdots&\\
	A_{\alpha111}(A_{\alpha111})=1,
	\end{cases}&
\end{flalign*}
\begin{flalign*}
	&\begin{cases}
		A_{1110}(A_{\alpha111})=A_{2111}(A_{\alpha111})=1,&\\
		A_{3110}(A_{\alpha111})=A_{3211}(A_{\alpha111})=0,&\\
			\quad\quad\quad\quad\vdots&\\
		A_{(\alpha-1)110}(A_{\alpha111})=q^2A_{(\alpha-1)211}(A_{\alpha111})=0,
	\end{cases}&
\end{flalign*}

\begin{flalign*}
	&\begin{cases}
		A_{2100}(A_{\alpha111})=q^2(q^3+1)A_{2111}(A_{\alpha111})=q^2(q^3+1),&\\
		\quad\quad\quad\quad\vdots&\\
		A_{\alpha100}(A_{\alpha111})=q^2(q^3+1)A_{\alpha111}(A_{\alpha111})=q^2(q^3+1),
	\end{cases}&
\end{flalign*}

\begin{flalign*}
	&\begin{cases}
		A_{1000}(A_{\alpha111})=(q^3+1)A_{2111}(A_{\alpha111})=(q^3+1),&\\
		A_{3000}(A_{\alpha111})=q^4(q+1)A_{3211}(A_{\alpha111})=0,&\\
		\quad\quad\quad\quad\vdots&\\
		A_{(\alpha-1)000}(A_{\alpha111})=q^4(q+1)A_{(\alpha-1)211}(A_{\alpha111})=0.
	\end{cases}&
\end{flalign*}

Therefore, we have that
\begin{equation*}\begin{array}{ll} A_{0000}'(A_{\alpha111})&=\dfrac{\alpha}{2}(1+q)(1-q^2)(1+q^3)+(1+q)(1-q^2)+\dfrac{\alpha}{2}q^2(q^3+1)+(q^3+1)\\\\
		&=\dfrac{\alpha}{2}(q+1)(q^3+1)+2+q-q^2.
		\end{array}
	\end{equation*}

Now, note that $A_{100}(A_{111})=(q^3+1)A_{111}(A_{111})=q^3+1$, and $A_{1110}(A_{\alpha111})=1$. This implies that
\begin{equation*}\begin{array}{l}
	A_{100}(A_{111})A_{0000}'(A_{\alpha100})+(1+q)(1-q^2)A_{1110}(A_{\alpha111})\\
	=(q^3+1)\lbrace (q+1)\dfrac{\alpha}{2}+1 \rbrace+(1+q)(1-q^2)\\\\
	=(q^3+1)(q+1)\dfrac{\alpha}{2}+2+q-q^2.
	\end{array}
\end{equation*}
Therefore, \eqref{eq3.2.3} holds in this case.\\

\item ($L_{\beta\gamma\delta}=L_{300}$ case) Indeed, in $L_{\beta\gamma\delta}=L_{\nu00}$ case, \eqref{eq3.2.3} holds since all $A_{\lambda00}(A_{\nu00})$ are $1$ for all $\nu \geq \lambda$. But, we will compute $A_{0000}'(A_{\alpha300})$ and $A_{0000}'(A_{\alpha300})-A_{0000}'(A_{\alpha100})$ for later use. By \eqref{eq3.2.4} we have that
\begin{equation*}\begin{array}{ll}
		A_{0000}'(A_{\alpha300})&=A_{1000}(A_{\alpha300})+\dots+A_{(\alpha-1)000}(A_{\alpha300})\\
		&+(1+q)\lbrace A_{2100}(A_{\alpha300})+\dots+A_{\alpha100}(A_{\alpha300})\\
		&+A_{3200}(A_{\alpha300})+A_{5200}(A_{\alpha300})+ \dots+A_{(\alpha-1)200}(A_{\alpha300})\\
		&+A_{4300}(A_{\alpha300})+A_{6200}(A_{\alpha300})+ \dots+A_{\alpha300}(A_{\alpha300})\rbrace
	\end{array}
\end{equation*}
Now, let us compute the following representation densities by using Corollary \ref{corollary2.20}.
\begin{flalign*}
	&\begin{cases}
		A_{4300}(A_{\alpha300})=A_{1}(A_{\alpha-3})=1,&\\
		\quad\quad\quad\quad\vdots&\\
		A_{\alpha300}(A_{\alpha300})=1,
	\end{cases}&
\end{flalign*}

\begin{flalign*}
	&\begin{cases}
		A_{3200}(A_{\alpha300})=A_{10}(A_{(\alpha-2)1})=A_{21}(A_{(\alpha-2)1})=A_{10}(A_{(\alpha-3)0})=1,&\\
		A_{5200}(A_{\alpha300})=A_{30}(A_{(\alpha-2)1})=q^2A_{32}(A_{(\alpha-2)1})=0,&\\
		\quad\quad\quad\quad\vdots&\\
		A_{(\alpha-1)200}(A_{\alpha300})=A_{(\alpha-3)0}(A_{(\alpha-2)1})=q^2A_{(\alpha-3)2}(A_{(\alpha-2)1})=0,
	\end{cases}&
\end{flalign*}

\begin{flalign*}
	&\begin{cases}
		A_{2100}(A_{\alpha300})=A_{10}(A_{(\alpha-1)2})=1,&\\
		A_{4100}(A_{\alpha300})=A_{30}(A_{(\alpha-1)2})=q^2A_{32}(A_{(\alpha-1)2})=q^2A_{10}(A_{(\alpha-3)0})=q^2,&\\
		\quad\quad\quad\quad\vdots&\\
		A_{\alpha100}(A_{\alpha300})=A_{(\alpha-1)0}(A_{(\alpha-1)2})=q^2A_{(\alpha-1)2}(A_{(\alpha-1)2})=q^2,
			\end{cases}&
	\end{flalign*}

\begin{flalign*}
	&\begin{cases}
		A_{1000}(A_{\alpha300})=A_{10}(A_{\alpha3})=1,&\\
		A_{3000}(A_{\alpha300})=A_{30}(A_{\alpha3})=q^2A_{32}(A_{\alpha3})=q^2A_{10}(A_{(\alpha-2)1})=q^2,&\\
		A_{5000}(A_{\alpha300})=A_{50}(A_{\alpha3})=q^2A_{52}(A_{\alpha3})&\\
		\quad\quad\quad\quad\quad\quad\quad\quad\quad=q^2A_{30}(A_{(\alpha-2)1})=q^4A_{32}(A_{(\alpha-2)1})=0,&\\
		\quad\quad\quad\quad\vdots&\\
		A_{(\alpha-1)000}(A_{\alpha300})=A_{(\alpha-1)0}(A_{\alpha3})=q^2A_{(\alpha-1)2}(A_{\alpha3})&\\
		\quad\quad\quad\quad\quad\quad\quad\quad\quad=q^2A_{(\alpha-3)0}(A_{(\alpha-2)1})=q^4A_{(\alpha-3)2}(A_{(\alpha-2)1})=0
	\end{cases}&
\end{flalign*}
Therefore, we have that
\begin{equation*}
	\begin{array}{ll}
		A_{0000}'(A_{\alpha300})&=1+q^2+(1+q)q^2\lbrace\dfrac{\alpha}{2}-1\rbrace+(1+q)+(1+q)+(1+q)\lbrace\dfrac{\alpha}{2}-1\rbrace\\
		&=(q+1)(q^2+1)\dfrac{\alpha}{2}-q^3+q+2.
		\end{array}
\end{equation*}

Also, we have that
\begin{equation*}
		A_{0000}'(A_{\alpha300})-A_{0000}'(A_{\alpha100})=q^2(q+1)\dfrac{\alpha}{2}-q^3+q+1.
\end{equation*}

\quad\\

\item ($L_{\beta\gamma\delta}=L_{311})$ case) By \eqref{eq3.2.4} we have that
\begin{equation*}\begin{array}{ll}
		&A_{0000}'(A_{\alpha311})\\
		&=A_{1000}(A_{\alpha311})+\dots+A_{(\alpha-1)000}(A_{\alpha311})\\\\
		&+(1+q)\lbrace A_{2100}(A_{\alpha311})+\dots+A_{\alpha100}(A_{\alpha311})\\
		&\quad\quad\quad\quad+A_{3200}(A_{\alpha311})+A_{5200}(A_{\alpha311})+ \dots+A_{(\alpha-1)200}(A_{\alpha311})\\
		&\quad\quad\quad\quad+A_{4300}(A_{\alpha311})+A_{6200}(A_{\alpha311})+ \dots+A_{\alpha300}(A_{\alpha311})\rbrace\\\\
		&+(1+q)(1-q^2)\lbrace A_{1110}(A_{\alpha311})+\dots+A_{(\alpha-1)110}(A_{\alpha311})\\
		&\quad\quad\quad\quad\quad\quad\quad+A_{2210}(A_{\alpha311})+\dots+A_{\alpha210}(A_{\alpha311})\\
		&\quad\quad\quad\quad\quad\quad\quad+A_{3310}(A_{\alpha311})+\dots+A_{(\alpha-1)310}(A_{\alpha311})\rbrace\\\\
		
		&+(1+q)(1-q^2)(1+q^3)\lbrace A_{2111}(A_{\alpha311})+\dots+A_{\alpha111}(A_{\alpha311})\\
		&\quad\quad\quad\quad\quad\quad\quad\quad\quad+A_{3211}(A_{\alpha311})+\dots+A_{(\alpha-1)211}(A_{\alpha311})\\
		&\quad\quad\quad\quad\quad\quad\quad\quad\quad+A_{4311}(A_{\alpha311})+\dots+A_{\alpha311}(A_{\alpha311})\rbrace
	\end{array}
\end{equation*}

Now, let us compute the following representation densities by using Corollary \ref{corollary2.20}.
\begin{flalign*}
	&\begin{cases}
		A_{4311}(A_{\alpha311})=1,&\\
		\qquad\qquad \vdots&\\
		A_{\alpha311}(A_{\alpha311})=1,
		\end{cases}&
\end{flalign*}

\begin{flalign*}
&\begin{cases}
	A_{3211}(A_{\alpha311})=A_{10}(A_{(\alpha-2)1})=1,&\\
	A_{5211}(A_{\alpha311})=A_{30}(A_{(\alpha-2)1})=q^2A_{32}(A_{(\alpha-2)1})=0,&\\
			\qquad\qquad\vdots&\\
	A_{(\alpha-1)211}=0 \text{ (since }A_{(\alpha-1)211}(A_{\alpha311}) \leq A_{5211}(A_{\alpha311})\text{)},
	\end{cases}&
\end{flalign*}

\begin{flalign*}
&\begin{cases}
A_{2111}(A_{\alpha311})=A_{10}(A_{(\alpha-1)2})=1,&\\
	A_{4111}(A_{\alpha311})=A_{30}(A_{(\alpha-1)2})=q^2A_{32}(A_{(\alpha-1)2})=q^2,&\\
	\qquad\qquad\vdots&\\
	A_{\alpha111}(A_{\alpha311})=A_{(\alpha-1)0}(A_{(\alpha-1)2})=q^2A_{(\alpha-1)2}(A_{(\alpha-1)2})=q^2,
	\end{cases}&
\end{flalign*}

\begin{flalign*}
&\begin{cases} A_{3310}(A_{\alpha311})=q^4A_{3321}(A_{\alpha311})=0,&\\
		\qquad\qquad\vdots&\\
		A_{(\alpha-1)310}(A_{\alpha311})=0,
		\end{cases}&
\end{flalign*}

\begin{flalign*}
	&\begin{cases}	A_{2210}(A_{\alpha311})=q^2(q^2-q+1)A_{2221}(A_{\alpha311})=0,&\\
		\qquad\qquad\vdots&\\
		A_{\alpha210}(A_{\alpha311})=0,
		\end{cases}&
\end{flalign*}

\begin{flalign*}
	&\begin{cases}	A_{1110}(A_{\alpha311})=A_{2111}(A_{\alpha311})=1,&\\
	A_{3110}(A_{\alpha311})=q^2A_{3211}(A_{\alpha311})=q^2,&\\
	A_{5110}(A_{\alpha311})=q^2A_{5211}(A_{\alpha311})=0&\\
	\qquad\qquad\vdots&\\
	A_{(\alpha-1)110}(A_{\alpha311})=0,
	\end{cases}&
\end{flalign*}

\begin{flalign*}
	&\begin{cases}	A_{4300}(A_{\alpha311})=q^4(q+1)A_{4311}(A_{\alpha311})=q^4(q+1),&\\
	\qquad\qquad\vdots&\\
		A_{\alpha300}(A_{\alpha311})=q^4(q+1)A_{\alpha311}(A_{\alpha311})=q^4(q+1),
	\end{cases}&
\end{flalign*}
\begin{flalign*}
	&\begin{cases}	A_{3200}(A_{\alpha311})=q^4(q+1)A_{3211}(A_{\alpha311})=q^4(q+1),&\\
		A_{5200}(A_{\alpha311})=q^4(q+1)A_{5211}(A_{\alpha311})=0,&\\
		\qquad\qquad\vdots&\\
		A_{(\alpha-1)200}(A_{\alpha311})=0,
	\end{cases}&
\end{flalign*}
\begin{flalign*}
	&\begin{cases}	A_{2100}(A_{\alpha311})=q^2(q^3+1)A_{2111}(A_{\alpha311})=q^2(q^3+1),&\\
		A_{4100}(A_{\alpha311})=q^2(q^3+1)A_{4111}(A_{\alpha311})=q^4(q^3+1),&\\
		\qquad\qquad\vdots&\\
		A_{\alpha100}(A_{\alpha311})=q^2(q^3+1)A_{\alpha111}(A_{\alpha311})=q^4(q^3+1),
	\end{cases}&
\end{flalign*}
\begin{flalign*}
	&\begin{cases}	A_{1000}(A_{\alpha311})=(q^3+1)A_{2111}(A_{\alpha311})=(q^3+1),&\\
		A_{3000}(A_{\alpha311})=q^4(q+1)A_{3211}(A_{\alpha311})=q^4(q+1),&\\
		A_{5000}(A_{\alpha311})=q^4(q+1)A_{5211}(A_{\alpha311})=0,&\\
		\qquad\qquad\vdots&\\
		A_{(\alpha-1)000}(A_{\alpha311})=q^4(q+1)A_{(\alpha-1)211}(A_{\alpha311})=0,
	\end{cases}&
\end{flalign*}

Therefore, we have that
\begin{equation*}\begin{array}{l}
	A_{0000}'(A_{\alpha311})\\\\
	=(q^3+1)+q^4(q+1)\\\\
	+(1+q) \lbrace (\dfrac{\alpha}{2}-1)q^4(q^3+1)+q^2(q^3+1)+q^4(q+1)+(\dfrac{\alpha}{2}-1)q^4(q+1)\rbrace\\\\
	
	+(1+q)(1-q^2)\lbrace q^2+1 \rbrace\\\\
	
	+(1+q)(1-q^2)(1+q^3)\lbrace (\dfrac{\alpha}{2}-1)q^2+1+1+(\dfrac{\alpha}{2}-1)\rbrace\\\\
	
	=\dfrac{\alpha}{2}(q+1)(q^5+q^4+q^3+1)-(q+1)(q^5-q^3+q-3).
	\end{array}
\end{equation*}

On the other hand, we have that
\begin{flalign*}
	&\begin{cases}	
		A_{100}(A_{311})=(q^3+1)A_{111}(A_{311})=(q^3+1),&\\
		A_{300}(A_{311})=q^2(q+1)A_{311}(A_{311})=q^2(q+1),&\\
		A_{111}(A_{111})=1,&\\
		A_{311}(A_{311})=1,
	\end{cases}&
\end{flalign*}

and
\begin{flalign*}
	&\begin{cases}	
		A_{0000}'(A_{\alpha100})=(q+1)\dfrac{\alpha}{2}+1,&\\
		A_{0000}'(A_{\alpha300})-A_{0000}'(A_{\alpha100})=q^2(q+1)\dfrac{\alpha}{2}-q^3+q+1,&\\
		A_{1110}(A_{\alpha111})=1,&\\
		A_{1110}(A_{\alpha311})=1.
	\end{cases}&
\end{flalign*}

These imply that
\begin{equation*}
	\begin{array}{l}
		A_{100}(A_{311})A_{0000}'(A_{\alpha100})+A_{300}(A_{311})\lbrace A_{0000}'\lbrace A_{\alpha300})-A_{0000}'(A_{\alpha100})\rbrace\\
		+(1+q)(1-q^2)A_{111}(A_{311})A_{1110}(A_{\alpha111})+(1+q)(1-q^2)A_{311}(A_{311})A_{1110}(A_{\alpha311})\\
		=(q^3+1)\lbrace (q+1)\dfrac{\alpha}{2}+1 \rbrace+ q^2(q+1)\lbrace q^2(q+1)\dfrac{\alpha}{2}-q^3+q+1 \rbrace\\
		+(1+q)(1-q^2)+(1+q)(1-q^2)\\\\
		=\dfrac{\alpha}{2}(q+1)(q^5+q^4+q^3+1)-(q+1)(q^5-q^3+q-3).
	\end{array}
\end{equation*}

Therefore, \eqref{eq3.2.3} holds in this case.\\

\item ($L_{\beta\gamma\delta}=L_{210})$ case) By \eqref{eq3.2.4} we have that
\begin{equation*}\begin{array}{ll}
		&A_{0000}'(A_{\alpha210})\\
		&=A_{1000}(A_{\alpha210})+\dots+A_{(\alpha-1)210}(A_{\alpha210})\\\\
		&+(1+q)\lbrace A_{2100}(A_{\alpha210})+\dots+A_{\alpha100}(A_{\alpha210})\\
		&\quad\quad\quad\quad+A_{3200}(A_{\alpha210})+A_{5200}(A_{\alpha210})+ \dots+A_{(\alpha-1)200}(A_{\alpha210})\\\\
		&+(1+q)(1-q^2)\lbrace A_{1110}(A_{\alpha210})+\dots+A_{(\alpha-1)110}(A_{\alpha210})\\
		&\quad\quad\quad\quad\quad\quad\quad+A_{2210}(A_{\alpha210})+\dots+A_{\alpha210}(A_{\alpha210})\rbrace
	\end{array}
\end{equation*}
Now, let us compute the following representation densities by using Corollary \ref{corollary2.20}.
\begin{flalign*}
	&\begin{cases}	A_{2210}(A_{\alpha210})=A_{2}(A_{\alpha})=1,&\\
		\qquad\qquad\vdots&\\
	A_{\alpha210}(A_{\alpha210})=1,&
	\end{cases}&
\end{flalign*}
\begin{flalign*}
	&\begin{cases}	
		A_{1110}(A_{\alpha210})=A_{111}(A_{\alpha21})=A_{000}(A_{(\alpha-1)10})=A_{00}(A_{(\alpha-1)1})=(q+1),&\\
		A_{3110}(A_{\alpha210})=A_{311}(A_{\alpha21})=A_{20}(A_{(\alpha-1)1})=q(q-1)A_{22}(A_{(\alpha-1)1})=0,&\\
		\qquad\qquad\vdots&\\
		A_{(\alpha-1)110}(A_{\alpha210})=0,&
	\end{cases}&
\end{flalign*}
\begin{flalign*}
	&\begin{cases}	
		A_{3200}(A_{\alpha210})=A_{320}(A_{\alpha21})=q^3(q-1)A_{322}(A_{\alpha21})=0,&\\
			\qquad\qquad\vdots&\\
			A_{(\alpha-1)200}(A_{\alpha2100})=0,&
	\end{cases}&
\end{flalign*}
\begin{flalign*}
	&\begin{cases}	
		A_{2100}(A_{\alpha210})=A_{210}(A_{\alpha21})=q(q-1)A_{221}(A_{\alpha21})=q(q-1),&\\
		A_{4100}(A_{\alpha210})=A_{410}(A_{\alpha21})=q^2A_{421}(A_{\alpha21})=q^2,&\\
		\qquad\qquad\vdots&\\
		A_{\alpha100}(A_{\alpha210})=A_{\alpha10}(A_{\alpha21})=q^2A_{\alpha21}(A_{\alpha21})=q^2,
	\end{cases}&
\end{flalign*}
\begin{flalign*}
	&\begin{cases}	
		A_{1000}(A_{\alpha210})=A_{100}(A_{\alpha21})=(q^3+1)A_{111}(A_{\alpha21})-qA_{221}(A_{\alpha21})&\\
		\qquad\qquad\quad\quad=(q^3+1)(q+1)-q     ,&\\
		A_{3000}(A_{\alpha210})=A_{300}(A_{\alpha21})=q^2(q+1)A_{311}(A_{\alpha21})-q^5A_{322}(A_{\alpha21})=0,&\\
		\qquad\qquad\vdots&\\
		A_{(\alpha-1)000}(A_{\alpha210})=0.
	\end{cases}&
\end{flalign*}

Therefore, we have that
\begin{equation*}
	\begin{array}{ll}
		A_{0000}'(A_{\alpha210})&=(q^3+1)(q+1)-q+(1+q)\lbrace(\dfrac{\alpha}{2}-1)q^2+q(q-1)\rbrace\\
		&+(1+q)(1-q^2)\lbrace (q+1)+\dfrac{a}{2}\rbrace\\
		&=\dfrac{a}{2}(q+1)-q^3-q^2+q+2.
		\end{array}
\end{equation*}

On the other hand, we have that $A_{100}(A_{210})=A_{10}(A_{21})=1$, and hence
\begin{equation*}\begin{array}{l}
	A_{100}(A_{210})A_{0000}'(A_{\alpha100})+(1-q^2)A_{210}(A_{210})A_{1110}(A_{\alpha210})\\
	=(q+1)\dfrac{\alpha}{2}+1+(1-q^2)(q+1)=\dfrac{\alpha}{2}(q+1)-q^3-q^2+q+2.
	\end{array}
\end{equation*}
This shows that \eqref{eq3.2.3} holds in this case.\\\\

\item ($L_{\beta\gamma\delta}=L_{421})$ case) By \eqref{eq3.2.4} we have that
\begin{equation}\label{eq3.2.5}\begin{array}{ll}
		&A_{0000}'(A_{\alpha421})\\
		&=A_{1000}(A_{\alpha421})+\dots+A_{(\alpha-1)000}(A_{\alpha421})\\\\
		&+(1+q)\lbrace A_{2100}(A_{\alpha421})+\dots+A_{\alpha100}(A_{\alpha421})\\
		&\quad\quad\quad\quad+A_{3200}(A_{\alpha421})+A_{5200}(A_{\alpha421})+ \dots+A_{(\alpha-1)200}(A_{\alpha421})\\
		&\quad\quad\quad\quad+A_{4300}(A_{\alpha421})+A_{6200}(A_{\alpha421})+ \dots+A_{\alpha300}(A_{\alpha421})\\
		&\quad\quad\quad\quad+A_{5400}(A_{\alpha421})+A_{7400}(A_{\alpha421})+ \dots+A_{(\alpha-1)400}(A_{\alpha421})\rbrace\\
\end{array}
\end{equation}		
\begin{equation*}\begin{array}{ll}				
		&+(1+q)(1-q^2)\lbrace A_{1110}(A_{\alpha421})+\dots+A_{(\alpha-1)110}(A_{\alpha421})\\
		&\quad\quad\quad\quad\quad\quad\quad+A_{2210}(A_{\alpha421})+\dots+A_{\alpha210}(A_{\alpha421})\\
		&\quad\quad\quad\quad\quad\quad\quad+A_{3310}(A_{\alpha421})+\dots+A_{(\alpha-1)310}(A_{\alpha421})\\
		&\quad\quad\quad\quad\quad\quad\quad+A_{4410}(A_{\alpha421})+\dots+A_{\alpha410}(A_{\alpha421})\\
		&\quad\quad\quad\quad\quad\quad\quad+A_{3220}(A_{\alpha421})+\dots+A_{(\alpha-1)220}(A_{\alpha421})\\
		&\quad\quad\quad\quad\quad\quad\quad+A_{4320}(A_{\alpha421})+\dots+A_{\alpha320}(A_{\alpha421})\\
		&\quad\quad\quad\quad\quad\quad\quad+A_{5420}(A_{\alpha421})+\dots+A_{(\alpha-1)420}(A_{\alpha421})
		\rbrace\quad\quad\quad\quad\quad
\end{array}
\end{equation*}		
\begin{equation*}\begin{array}{ll}		
		&+(1+q)(1-q^2)(1+q^3)\lbrace A_{2111}(A_{\alpha421})+\dots+A_{\alpha111}(A_{\alpha421})\\
		&\quad\quad\quad\quad\quad\quad\quad\quad\quad\quad\quad+A_{3211}(A_{\alpha421})+\dots+A_{(\alpha-1)211}(A_{\alpha421})\\
		&\quad\quad\quad\quad\quad\quad\quad\quad\quad\quad\quad+A_{4311}(A_{\alpha421})+\dots+A_{\alpha311}(A_{\alpha421})\\
		&\quad\quad\quad\quad\quad\quad\quad\quad\quad\quad\quad+A_{5411}(A_{\alpha421})+\dots+A_{(\alpha-1)411}(A_{\alpha421})\\
		&\quad\quad\quad\quad\quad\quad\quad\quad\quad\quad\quad+A_{2221}(A_{\alpha421})+\dots+A_{\alpha221}(A_{\alpha421})\\
		&\quad\quad\quad\quad\quad\quad\quad\quad\quad\quad\quad+A_{3321}(A_{\alpha421})+\dots+A_{(\alpha-1)321}(A_{\alpha421})\\
		&\quad\quad\quad\quad\quad\quad\quad\quad\quad\quad\quad+A_{4421}(A_{\alpha421})+\dots+A_{\alpha421}(A_{\alpha421}).
		\rbrace\quad\quad
	\end{array}
\end{equation*}

Now, let us compute the following representation densities by using Corollary \ref{corollary2.20}.
\begin{flalign*}
	&\begin{cases}	
		A_{4421}(A_{\alpha421})=A_{4}(A_{\alpha})=1,&\\
		\qquad\qquad\vdots&\\
		A_{\alpha421}(A_{\alpha421})=1,
	\end{cases}&
\end{flalign*}
\begin{flalign*}
	&\begin{cases}	
		A_{3321}(A_{\alpha421})=A_{00}(A_{(\alpha-3)1})=(q+1),&\\
		A_{5321}(A_{\alpha421})=A_{53}(A_{\alpha4})=A_{20}(A_{(\alpha-3)1})=q(q-1)A_{22}(A_{(\alpha-3)1})=0,&\\
		\qquad\qquad\vdots&\\
		A_{(\alpha-1)321}(A_{\alpha421})=0,
	\end{cases}&
\end{flalign*}
\begin{flalign*}
	&\begin{cases}	
		A_{2221}(A_{\alpha421})=A_{00}(A_{(\alpha-2)2})=(q+1)A_{11}(A_{(\alpha-2)2})-qA_{22}(A_{(\alpha-2)2})&\\
		\quad\quad\quad\quad\quad=(q+1)A_{00}(A_{(\alpha-3)1})-q=(q^2+q+1),&\\
		A_{4221}(A_{\alpha421})=A_{20}(A_{(\alpha-2)2})=q(q-1)A_{22}(A_{(\alpha-2)2})=q(q-1),&\\
		A_{6221}(A_{\alpha421})=A_{40}(A_{(\alpha-2)2})=q^2A_{42}(A_{(\alpha-2)2})=q^2,&\\
		\qquad\qquad\vdots&\\
		A_{\alpha221}(A_{\alpha421})=A_{(\alpha-2)0}(A_{(\alpha-2)2})=q^2A_{(\alpha-2)2}(A_{(\alpha-2)2})=q^2,
	\end{cases}&
\end{flalign*}

\begin{flalign*}
	&\begin{cases}
		A_{5411}(A_{\alpha421})=A_{430}(A_{(\alpha-1)31})=q^4A_{432}(A_{(\alpha-1)31})=0,&\\
		\qquad\qquad\vdots&\\
		A_{(\alpha-1)411}(A_{\alpha421})=0,
	\end{cases}&
\end{flalign*}
\begin{flalign*}
	&\begin{cases}
		A_{4311}(A_{\alpha421})=A_{320}(A_{(\alpha-1)31})=q^3(q-1)_{332}(A_{(\alpha-1)31})=0,&\\
		\qquad\qquad\vdots&\\
		A_{\alpha311}(A_{\alpha421})=0,
	\end{cases}&
\end{flalign*}
\begin{flalign*}
	&\begin{cases}
		A_{3211}(A_{\alpha421})=A_{210}(A_{(\alpha-1)31})=q(q-1)A_{221}(A_{(\alpha-1)31})&\\
		\quad\quad\quad\quad\quad\quad\quad\quad=q(q-1)A_{00}(A_{(\alpha-3)1})=q(q^2-1),&\\
		A_{5211}(A_{\alpha421})=A_{521}(A_{\alpha42})=A_{410}(A_{(\alpha-1)31})=q^2A_{421}(A_{(\alpha-1)31})&\\
		\quad\quad\quad\quad\quad\quad=q^2A_{20}(A_{(\alpha-3)1})=0 ,&\\
		\qquad\qquad\vdots&\\
		A_{(\alpha-1)211}(A_{\alpha421})=0,
	\end{cases}&
\end{flalign*}
\begin{flalign*}
	&\begin{cases}
		A_{2111}(A_{\alpha421})=A_{100}(A_{(\alpha-1)31})=(q^3+1)A_{111}(A_{(\alpha-1)31})-qA_{221}(A_{(\alpha-1)31})&\\
		\quad\quad\quad\quad\quad\quad\quad\quad\quad\quad\quad\quad\quad=(q^3+1)A_{00}(A_{(\alpha-2)2})-qA_{00}(A_{(\alpha-3)1})&\\
		\quad\quad\quad\quad\quad\quad\quad\quad\quad\quad\quad\quad\quad=(q^3+1)(q^2+q+1)-q(q+1),&\\
		A_{4111}(A_{\alpha421})=A_{300}(A_{(\alpha-1)31})=q^2(q+1)A_{311}(A_{(\alpha-1)31})&\\
		\quad\quad\quad\quad=q^2(q+1)A_{20}(A_{(\alpha-2)2})=q^3(q^2-1)A_{22}(A_{(\alpha-2)2})=q^3(q^2-1),
			\end{cases}&
	\end{flalign*}
\begin{flalign*}
	&\begin{cases}
		A_{6111}(A_{\alpha421})=A_{500}(A_{(\alpha-1)31})=q^2(q+1)A_{511}(A_{(\alpha-1)31})&\\
		\quad\quad\quad\quad\quad=q^2(q+1)A_{40}(A_{(\alpha-2)2})=q^4(q+1)A_{42}(A_{(\alpha-2)})=q^4(q+1),&\\
		\qquad\qquad\vdots&\\
		A_{\alpha111}(A_{\alpha421})=A_{(\alpha-1)00}(A_{(\alpha-1)31})=q^2(q+1)A_{(\alpha-1)11}(A_{(\alpha-1)31})
		&\\
		\quad\quad=q^2(q+1)A_{(\alpha-2)0}(A_{(\alpha-2)2})=q^4(q+1)A_{(\alpha-2)2}(A_{(\alpha-2)})=q^4(q+1),
	\end{cases}&
\end{flalign*}
\begin{flalign*}
	&\begin{cases}
		A_{5420}(A_{\alpha421})=q^5(q-1)A_{5422}(A_{\alpha421})=0,&\\
		\qquad\qquad\vdots&\\
		A_{(\alpha-1)420}(A_{\alpha421})=0,
	\end{cases}&
\end{flalign*}
\begin{flalign*}
	&\begin{cases}
		A_{4320}(A_{\alpha421})=q^5(q-1)A_{4322}(A_{\alpha421})=0,&\\
		\qquad\qquad\vdots&\\
		A_{\alpha320}(A_{\alpha421})=0,
	\end{cases}&
\end{flalign*}	
\begin{flalign*}
	&\begin{cases}
		A_{3220}(A_{\alpha421})=q^4(q^2-q+1))A_{3222}(A_{\alpha421})=0,&\\
		\qquad\qquad\vdots&\\
		A_{(\alpha-1)220}(A_{\alpha421})=0,
	\end{cases}&
\end{flalign*}	
\begin{flalign*}
	&\begin{cases}
		A_{4410}(A_{\alpha421})=q^4A_{4421}(A_{\alpha421})=q^4,&\\
		\qquad\qquad\vdots&\\
		A_{\alpha410}(A_{\alpha421})=q^4A_{\alpha421}(A_{\alpha421})=q^4,
	\end{cases}&
\end{flalign*}	
\begin{flalign*}
	&\begin{cases}
		A_{3310}(A_{\alpha421})=q^4A_{3321}(A_{\alpha421})=q^4(q+1),&\\
		A_{5310}(A_{\alpha421})=q^4A_{5321}(A_{\alpha421})=0,&\\
		\qquad\qquad\vdots&\\
		A_{(\alpha-1)310}(A_{\alpha421})=0,
	\end{cases}&
\end{flalign*}	
\begin{flalign*}
	&\begin{cases}
		A_{2210}(A_{\alpha421})=q^2(q^2-q+1)A_{2221}(A_{\alpha421})=q^2(q^2-q+1)(q^2+q+1),&\\
		A_{4210}(A_{\alpha421})=q^3(q-1)A_{4221}(A_{\alpha421})=q^4(q-1)^2,&\\
		A_{6210}(A_{\alpha421})=q^3(q-1)A_{6221}(A_{\alpha421})=q^5(q-1),&\\
		\qquad\qquad\vdots&\\
		A_{\alpha210}(A_{\alpha421})=q^3(q-1)A_{\alpha221}(A_{\alpha421})=q^5(q-1),
	\end{cases}&
\end{flalign*}
\begin{flalign*}
	&\begin{cases}
		A_{1110}(A_{\alpha421})=A_{2111}(A_{\alpha421})=(q^3+1)(q^2+q+1)-q(q+1),&\\
		A_{3110}(A_{\alpha421})=q^2A_{3211}(A_{\alpha421})=q^3(q^2-1),&\\
		A_{5110}(A_{\alpha421})=q^2A_{5211}(A_{\alpha421})=0,&\\
		\qquad\qquad\vdots&\\
		A_{(\alpha-1)110}(A_{\alpha421})=0,
	\end{cases}&
\end{flalign*}	
\begin{flalign*}
	&\begin{cases}
		A_{5400}(A_{\alpha421})=q^4(q+1)A_{5411}(A_{\alpha421})=0,&\\
		\qquad\qquad\vdots&\\
		A_{(\alpha-1)400}(A_{\alpha421})=0,
	\end{cases}&
\end{flalign*}	
\begin{flalign*}
	&\begin{cases}
		A_{4300}(A_{\alpha421})=q^4(q+1)A_{4311}(A_{\alpha421})=0,&\\
		\qquad\qquad\vdots&\\
		A_{\alpha300}(A_{\alpha421})=0,
	\end{cases}&
\end{flalign*}	
\begin{flalign*}
	&\begin{cases}
		A_{3200}(A_{\alpha421})=q^4(q+1)A_{3211}(A_{\alpha421})=q^5(q+1)(q^2-1),&\\
		A_{5200}(A_{\alpha421})=q^4(q+1)A_{5211}(A_{\alpha421})=0,&\\
		\qquad\qquad\vdots&\\
		A_{(\alpha-1)200}(A_{\alpha421})=0,
	\end{cases}&
\end{flalign*}	
\begin{flalign*}
	&\begin{cases}
		A_{2100}(A_{\alpha421})=q^2(q^3+1)A_{2111}(A_{\alpha421})-q^3(q^2-q+1)A_{2221}(A_{\alpha421})&\\
		=q^2(q^3+1)\lbrace (q^3+1)(q^2+q+1)-q(q+1)\rbrace-q^3(q^2-q+1)(q^2+q+1),&\\
		A_{4100}(A_{\alpha421})=q^2(q^3+1)A_{4111}(A_{\alpha421})-q^5A_{4221}(A_{\alpha421})&\\
		\qquad\qquad\qquad=q^5(q^2-1)(q^3+1)-q^6(q-1),
		\end{cases}&
	\end{flalign*}	
\begin{flalign*}
	&\begin{cases}
		A_{6100}(A_{\alpha421})=q^2(q^3+1)A_{6111}(A_{\alpha421})-q^5A_{6221}(A_{\alpha421})&\\
		\qquad\qquad\qquad=q^6(q+1)(q^3+1)-q^7,&\\
		\qquad\qquad\vdots&\\
		A_{\alpha100}(A_{\alpha421})=q^2(q^3+1)A_{\alpha111}(A_{\alpha421})-q^5A_{\alpha221}(A_{\alpha421})&\\
		\qquad\qquad\qquad=q^6(q+1)(q^3+1)-q^7,
	\end{cases}&
\end{flalign*}	
\begin{flalign*}
	&\begin{cases}
		A_{1000}(A_{\alpha421})=(q^3+1)A_{2111}(A_{\alpha421})-q^3A_{2221}(A_{\alpha421})&\\
		\qquad\qquad\qquad=(q^3+1)\lbrace (q^3+1)(q^2+q+1)-q(q+1)\rbrace-q^3(q^2+q+1),&\\
		A_{3000}(A_{\alpha421})=q^4(q+1)A_{3211}(A_{\alpha421})=q^5(q+1)(q^2-1),
	\end{cases}&
\end{flalign*}	
\begin{flalign*}
	&\begin{cases}
		A_{5000}(A_{\alpha421})=q^4(q+1)A_{5211}(A_{\alpha421})=0,&\\
		\qquad\qquad\vdots&\\
		A_{(\alpha-1)000}(A_{\alpha421})=0.
	\end{cases}&
\end{flalign*}

Now, by \eqref{eq3.2.5}, we have
\begin{equation*}\begin{array}{ll}
	A_{0000}'(A_{\alpha421})&=(1+q)(1-q^2)(1+q^3)\\
	&\times \mathlarger{\mathlarger{\lbrace}} (\dfrac{\alpha}{2}-1)+(q+1)+(q^2+q+1)+q(q-1)+q^2(\dfrac{\alpha}{2}-2)\\
	&\quad+q(q^2-1)+(q^3+1)(q^2+q+1)-q(q+1)+q^3(q^2-1)\\
	&\quad+q^4(q+1)(\dfrac{\alpha}{2}-2) \mathlarger{\mathlarger{\rbrace}}
	\end{array}
\end{equation*}
\begin{equation*}\begin{array}{ll}
	\qquad\qquad\qquad&+(1+q)(1-q^2)\\
		&\times \mathlarger{\mathlarger{\lbrace}} q^4(\dfrac{\alpha}{2}-1)+q^4(q+1)+q^2(q^2-q+1)(q^2+q+1)+q^4(q-1)^2\\ &+q^5(q-1)(\dfrac{\alpha}{2}-2)+(q^3+1)(q^2+q+1)-q(q+1)+q^3(q^2-1)\mathlarger{\mathlarger{\rbrace}}
	\end{array}
\end{equation*}
\begin{equation*}\begin{array}{ll}
		\qquad\qquad\qquad&+(1+q)\\
		&\times \mathlarger{\mathlarger{\lbrace}} q^5(q+1)(q^2-1)+q^2(q^3+1)\lbrace (q^3+1)(q^2+q+1)-q(q+1)\rbrace\\
		&-q^3(q^2-q+1)(q^2+q+1)+q^5(q^2-1)(q^3+1)-q^6(q-1)\\
		&+\lbrace q^6(q+1)(q^3+1)-q^7 \rbrace (\dfrac{\alpha}{2}-2)
		\mathlarger{\mathlarger{\rbrace}}
	\end{array}
\end{equation*}
\begin{equation*}\begin{array}{ll}
		\qquad\qquad\qquad& +(q^3+1)\lbrace (q^3+1)(q^2+q+1)-q(q+1)\rbrace\\
		&-q^3(q^2+q+1)+q^5(q+1)(q^2-1) \qquad\qquad\qquad\qquad\qquad\qquad\qquad\qquad\qquad
		
	\end{array}
\end{equation*}
\begin{equation*}\begin{array}{ll}
		\qquad\qquad\qquad&= (1+q)(1-q^2)(1+q^3)\lbrace \dfrac{\alpha}{2}(q^5+q^4+q^2+1)-q^4+q^3+2\rbrace\\
		&+(1+q)(1-q^2)\lbrace \dfrac{\alpha}{2}(q^6-q^5+q^4)+3q^5+3q^4+q^2+1\rbrace\\
		&+(1+q)\lbrace \dfrac{\alpha}{2}(q^{10}+q^9+q^6)-q^9+q^8+q^7-q^6-q^5-q^3+q^2\rbrace\\\\
		&+2q^8+2q^7-q^5+q^3+1
	\end{array}
\end{equation*}
\begin{equation*}\begin{array}{ll}
		\qquad\qquad&=\dfrac{a}{2}(q+1)(q^4+q^3+1)-q^8-3q^7-3q^6-q^5+3q^4+2q^3-q^2+3q+4.
	\end{array}
\end{equation*}

On the other hand, we have that

\begin{flalign*}
&\begin{cases}
A_{100}(A_{421})=(q^3+1)A_{111}(A_{421})-qA_{221}(A_{421})=q^4+q^3+1,&\\
A_{0000}'(A_{\alpha100})=\dfrac{a}{2}(q+1)+1,&
\end{cases}&
\end{flalign*}
\begin{flalign*}
	&\begin{cases}
A_{300}(A_{421})=q^2(q+1)A_{311}(A_{421})=0,
\end{cases}&
\end{flalign*}
\begin{flalign*}
		&\begin{cases}
	A_{210}(A_{421})=q(q-1)A_{221}(A_{421})=q(q-1),&\\
	A_{1110}(A_{\alpha210})=A_{111}(A_{\alpha21})=A_{00}(A_{(\alpha-1)1})=(q+1),
\end{cases}&
\end{flalign*}
\begin{flalign*}
		&\begin{cases}
	A_{410}(A_{421})=q^2A_{421}(A_{421})=q^2,&\\
	A_{1110}(A_{\alpha410})=A_{111}(A_{\alpha41})=A_{00}(A_{(\alpha-1)3})=(q^3+q^2+q+1),
\end{cases}&
\end{flalign*}
\begin{flalign*}
	&\begin{cases}
	A_{111}(A_{421})=A_{00}(A_{31})=q+1,&\\
	A_{1110}(A_{\alpha111})=A_{2111}(A_{\alpha111})=1,
\end{cases}&
\end{flalign*}
\begin{flalign*}
		&\begin{cases}
	A_{221}(A_{421})=1,&\\
	A_{1110}(A_{\alpha221})=A_{2111}(A_{\alpha221})=A_{100}(A_{(\alpha-1)11})=(q^3+1)A_{111}(A_{(\alpha-1)11})=(q^3+1),
\end{cases}&
\end{flalign*}
\begin{flalign*}
		&\begin{cases}
	A_{311}(A_{421})=0,
\end{cases}&
\end{flalign*}
\begin{flalign*}
		&\begin{cases}
			A_{421}(A_{421})=1,&\\
	A_{1110}(A_{\alpha421})=A_{2111}(A_{\alpha421})=(q^3+1)(q^2+q+1)-q(q+1).
\end{cases}&
\end{flalign*}

Therefore, we have that
\begin{equation*}\begin{array}{l}
	A_{100}(A_{421})A'_{0000}(A_{\alpha100})+A_{300}(A_{421})\lbrace A_{0000}'(A_{\alpha300})-A_{0000}'(A_{\alpha100})\rbrace\\
	+(1-q^2)A_{210}(A_{421})A_{1110}(A_{\alpha210})+(1-q^2)A_{410}(A_{421})A_{1110}(A_{\alpha410})\\
	+(1+q)(1-q^2)A_{111}(A_{421})A_{1110}(A_{\alpha111})
	+(1+q)(1-q^2)A_{221}(A_{421})A_{1110}(A_{\alpha221})\\
	+(1+q)(1-q^2)A_{311}(A_{421})A_{1110}(A_{\alpha311})+(1+q)(1-q^2)A_{421}(A_{421})A_{1110}(A_{\alpha421})\\\\
	=\dfrac{\alpha}{2}(q+1)(q^4+q+1)-q^8-3q^7-3q^6-q^5+3q^4+2q^3-q^2+3q+4
	\end{array}
\end{equation*}
This shows that \eqref{eq3.2.3} holds in this case.
	\end{enumerate}
\end{example}

\bigskip

From the example \ref{example3.4}, we conjecture that the following formula holds in general.

\begin{conjecture}[$\CN^0_{E/F}(1,3)$]\label{conjecture} Assume that $\alpha > \beta \geq \gamma \geq \delta \geq 0$. Then we have
	\begin{equation*}\begin{array}{ll}
		A_{0000}'(A_{\alpha\beta\gamma\delta})&=A_{100}(A_{\beta\gamma\delta})A_{0000}'(A_{\alpha100})\\\\
		&+\mathlarger{\sum}_{\lambda \geq 3}A_{\lambda00}(A_{\beta\gamma\delta})\lbrace A_{0000}' (A_{\alpha\lambda00})-A_{0000}'(A_{\alpha(\lambda-2)00})\rbrace\\\\
		&+(1-q^2)\sum_{\lambda \geq \kappa \geq 1} A_{\lambda\kappa0}(A_{\beta\gamma\delta}) A_{1110}(A_{\alpha\lambda\kappa0})\\\\
		&+(1+q)(1-q^2)\sum_{\lambda \geq \kappa \geq \mu \geq 1} A_{\lambda\kappa\mu}(A_{\beta\gamma\delta}) A_{1110}(A_{\alpha\lambda\kappa\mu}).
		
		\end{array}
	\end{equation*}
\end{conjecture}

\begin{remark}\label{remark}
	This Conjecture \ref{conjecture} suggests that the arithmetic intersection numbers of special cycles on unitary Rapoport-Zink spaces $\CN_{E/F}^0(1,3)$ are closely related to the counting of vertex lattices. Let $L_{\beta\gamma\delta}$ be a lattice of rank $3$ with the hermitian matrix $A_{\beta\gamma\delta}$. By \cite[Theorem 4.2.1]{LZ}, $A_{0000}'(A_{\alpha\lambda00})$ terms in the conjecture are related to the horizontal part of $\CZ(L_{\beta\gamma\delta})$. Also, even though it is not clear yet (for example, we do not know why we have $(1+q)(1-q^2)$ in the sum instead of $(1-q^2)$), we believe that this conjecture has information about the multiplicities of the vertical components attached to vertex lattices of type $3$ in the special cycle $\CZ(L_{\beta\gamma\delta})$.
	
	Based on the fact that the irreducible components of the reduced subscheme of $\CN_{E/F}^h(1,n-1)$ are closely related to vertex lattices (\cite{VW}, \cite{Cho}), it is reasonable to ask the following quesiton: for the conjectural formula for the arithmetic intersection numbers of special cycles on $\CN_{E/F}^h(1,n-1)$ in \cite[Conjecture 3.17, Conjecture 3.25]{Cho2}, can we write it in terms of vertex lattice counting plus some additional parts like horizontal parts? We believe that this will give us important information about special cycles. We already have a version of the Cho-Yamauchi's formula in \cite[Conjecture 4.9]{Cho3}, and we have an efficient computational method by Theorem \ref{theorem2.5}. Therefore, we think that we have enough tools to consider this question. We hope to return to this in the future.
	\end{remark}

\end{document}